\documentclass[reqno, 11pt, twoside]{amsart}

\usepackage[margin=3.5cm]{geometry}

\usepackage{amsmath,amssymb,amsthm,bm,colonequals,graphicx,mathrsfs,mathtools,microtype,stmaryrd, amsfonts, dsfont,physics, comment, multicol}
\usepackage[shortlabels]{enumitem}
\usepackage[colorinlistoftodos]{todonotes}

\numberwithin{equation}{section}

\theoremstyle{plain}

\newtheorem{theorem}{Theorem}[section] 

\newtheorem{lemma}[theorem]{Lemma} 

\newtheorem{proposition}[theorem]{Proposition} 

\newtheorem{proposition-definition}[theorem]{Proposition-Definition}

\newtheorem{conjecture}[theorem]{Conjecture} 

\theoremstyle{definition}

\newtheorem{definition}[theorem]{Definition}

\theoremstyle{remark}

\newtheorem{remark}[theorem]{Remark}

\renewcommand{\abs}[1]{\lvert #1 \rvert} 

\renewcommand{\norm}[1]{\lvert #1 \rvert} 
\newcommand{\floor}[1]{\lfloor #1 \rfloor}

\newcommand{\belongs}{\subseteq}

\newcommand{\eps}{\epsilon}
\newcommand{\EE}{\mathbb{E}}
\newcommand{\defeq}{\colonequals}

\newcommand{\map}{\operatorname}
\newcommand{\mscr}{\mathscr}
\newcommand{\mcal}{\mathcal}
\newcommand{\set}[1]{\{#1\}}
\newcommand{\ts}{\tau}

\newcommand{\NN}{\mathbb{N}}
\newcommand{\ZZ}{\mathbb{Z}}
\newcommand{\QQ}{\mathbb{Q}}
\newcommand{\CC}{\mathbb{C}}

\newcommand{\FF}{\mathbb{F}}

\newcommand{\PP}{\mathbb{P}}

\DeclareMathOperator{\supp}{supp}

\DeclareMathOperator{\vol}{vol}

\DeclareMathOperator{\cha}{char}

\def\OK{\mathcal{O}}
\def\Omon{\OK^+}

\def\RR{\mathbb{R}}
\def\TT{\mathbb{T}}
\def\Sing{\mathfrak{S}}

\newcommand{\RcG}{\mathcal{R}_{\bm{c}}^{\map{G}}} 
\newcommand{\RcB}{\mathcal{R}_{\bm{c}}^{\map{B}}} 

\renewcommand{\hat}{\widehat}
\newcommand{\cc}{{\bm{c}}}

\DeclareMathOperator{\Char}{char}
\renewcommand{\leq}{\leqslant}
\renewcommand{\le}{\leqslant}
\renewcommand{\geq}{\geqslant}
\renewcommand{\ge}{\geqslant}

\renewcommand{\eps}{\varepsilon}

\usepackage[pagebackref=true]{hyperref}
\usepackage[alphabetic,initials]{amsrefs}

\title{Optimal sums of three cubes in $\mathbb{F}_q[t]$}

\author{Tim Browning, Jakob Glas, Victor Y. Wang}

\address{IST Austria\\
Am Campus 1\\
3400 Klosterneuburg\\
Austria}
\email{tdb@ist.ac.at, jakob.glas@ist.ac.at, victor.wang@ist.ac.at}

\subjclass[2010]{11D45 (11D25, 11G40, 11M50, 11P55, 11T55)}

\date{}

\begin{document}

\begin{abstract}
We use the circle method to prove that a density $1$ of elements in $\FF_q[t]$ are representable as a sum of three cubes
of essentially minimal degree
from $\FF_q[t]$, assuming
the Ratios Conjecture and 
that $\Char(\FF_q)>3$.
Roughly speaking, to do so, we 
upgrade an order of magnitude result to a full asymptotic formula that was conjectured by Hooley in the number field setting. 
\end{abstract}

\maketitle

\thispagestyle{empty}
 \setcounter{tocdepth}{1}
 {\small
 \begin{multicols}{2}
 \tableofcontents
 \end{multicols}
 }

\section{Introduction}

As in our previous paper \cite{BGW2024positive},
we are interested in solving the equations $x^3+y^3+z^3=k$ in $\FF_q[t]$ as efficiently as possible, for given $k\in \FF_q[t]$.
As observed by Serre and Vaserstein \cite{vaserstein1991sums}, this Diophantine equation is always soluble when 
when $\cha(\FF_q)\ne 3$ and $q\notin \{2,4,16\}$, 
but the degrees of $x,y,z\in \OK$ solving $x^3+y^3+z^3 = k$
are at least $\deg{k}$, whereas one might hope for solutions of smaller degree. 
Let $\OK=\FF_q[t]$.
For each $A\in \RR$, we define the set
\begin{equation}\label{def:SA}
S_A=\left\{k\in \OK: 
\begin{array}{l}
x^3+y^3+z^3=k\textnormal{ is soluble in $\OK$ with } \\
\max\{\deg{x},\deg{y},\deg{z}\} \le \tfrac{\deg{k}}{3} + A
\end{array}
\right\}.
\end{equation}
In \cite{BGW2024positive}, we proved that $S_A$ has positive lower density for $A\ge 0$,
assuming (R2) from the Ratios Conjecture \cite{BGW2024positive}*{Conjecture~3.6}.
At the expense of assuming a further instance of the Ratios Conjecture,
our main result is as follows.

\begin{theorem}\label{THM:main-intro}
Assume (R2) and the Ratios Conjecture~\ref{CNJ:RA1},
and suppose $\cha(\FF_q) > 3$.
Then the lower density of $S_A$ approaches $1$ as $A\to \infty$.
\end{theorem}

Conjecture~\ref{CNJ:RA1}~(RA1) concerns mean values of $1/L(s,\cc)$ over adelic boxes of vectors $\cc$.
We defer the details to \S~\ref{SEC:new-ratios-stuff}.
In Theorem~\ref{THM:main-intro}, one could simply assume a common generalisation of (RA1) and (R2), but this would obfuscate the paper.

Taking $A\to \infty$ is necessary for the conclusion of Theorem~\ref{THM:main-intro} to hold.
By adapting local density arguments of Diaconu \cite{diaconu2019admissible}*{\S~1} from $\ZZ$ to $\FF_q[t]$, we will prove the following unconditional result.

\begin{theorem}\label{THM:fixed-A}
Suppose $\cha(\FF_q) \ne 3$ and
fix $A\in \RR$.
Then $S_A$
has upper density $<1$.
\end{theorem}

The proof of Theorem~\ref{THM:main-intro} builds heavily on our work in \cite{BGW2024positive},
using the full force of the function field circle method. This will allow us to prove the following asymptotic formula, in the spirit of the Manin--Peyre conjecture. 
\begin{theorem}\label{Thm: Manin}
Suppose $\cha(\FF_q) > 3$.
    Let $w\colon K_\infty^6\to \RR$ be the weight function defined in Definition~\ref{DEFN:A-skew-weights} and let $F(\bm{x})=x_1^3+\cdots +x_6^3$. Assuming (R2) and Conjecture~\ref{CNJ:RA1}, for $P=t^d$ we have 
    \[
    \sum_{\substack{\bm{x}\in\OK^6\\ F(\bm{x})=0}}w(\bm{x}/P)=\sigma_\infty\Sing|P|^3 + \sum_{L\in \Upsilon}\sum_{\bm{x}\in \OK^6\cap L}w(\bm{x}/P)+o_w(|P|^3)
    \]
    as $|P|\to\infty$, where $\sigma_\infty$ and $\Sing$ are the singular integral and singular series defined in \eqref{Eq: singInt} and \eqref{Eq: SingSer}, respectively. Moreover,  $\Upsilon$ denotes the set of 3-dimensional $\FF_q(t)$-vector spaces defined over $\FF_q$ on which $F$ vanishes identically.
\end{theorem}
In some key ranges,
we rely on Conjecture~\ref{CNJ:RA1} to improve a certain $O(\abs{P}^3)$ bound to $o(\abs{P}^3)$, 
in the argument of \cite{BGW2024positive}.
Moreover, the scaling-invariant weight functions required for our counting argument
are subtler than those in \cite{BGW2024positive}, which already required care.
The precise weights we use are specified in Definition~\ref{DEFN:A-skew-weights}.

\subsection*{Acknowledgements}

We thank Alexandra Florea for discussions on cubic Gauss sums over function fields.
While working on this paper the first two authors were supported by a FWF grant (DOI 10.55776/P36278)
 and the third  author was supported by the European Union's Horizon 2020 research and innovation programme under the Marie Sk\l{}odowska-Curie Grant Agreement No.~101034413.

\section{Upper density of \texorpdfstring{$S_A$}{SA}}

Suppose $\cha(\FF_q) \ne 3$.
In this section we prove
Theorem  \ref{THM:fixed-A}.
We recall the notation $\hat{B}=q^B$, for any $B\in \RR$, that was adopted in our previous investigation \cite{BGW2024positive}. 
Define
$$
\mathbb{E}_r(T;S)= 
\frac{|r|}{\hat T}\#\left\{k\in S: |k|<\hat T,\; r\mid k\right\},
$$
for any subset $S\subset \OK$,  
any $T>0$ 
and any $r\in \Omon$, where we recall that $\Omon$ denotes the set of monic elements in $\OK$.
For a given choice of $A>0$ we need to prove that 
$$
\limsup_{T\to \infty} \mathbb{E}_1(T;S_A)<1.
$$

To begin with, given
$r\in \Omon$ such that $\deg(r)\leq \frac{T}{3}$, we have 
$$
\mathbb{E}_r(T;S_A)
\leq 
\frac{|r|}{\hat T}
\sum_{\substack{k\in \OK\\
|k|<\hat T\\  r\mid k}} \tilde r_{A}(k),
$$
where 
$$
\tilde r_{A}(k)=
\#\left\{(x,y,z)\in \OK^3: x^3+y^3+z^3=k, ~
\deg{x},\deg{y},\deg{z} \le \tfrac{\deg{k}}{3} + A\right\}.
$$
Breaking into residue classes modulo $r$, we obtain
$$
\sum_{\substack{k\in \OK\\
|k|<\hat T\\  r\mid k}} \tilde r_{A}(k)
\leq \sum_{\substack{
(u,v,w)\in (\OK/r\OK)^3\\  u^3+v^3+w^3\equiv 0 \bmod{r}}} 
S(u)S(v)S(w),
$$
where $S(n)$ is the number
$x\in \OK$ such that $\deg x\leq \tfrac{1}{3}T + A$ and $x\equiv n\bmod{r}$, for any $n \in \OK$. 
Clearly $S(n)\ll_A \hat T^{1/3}/|r|$, where the implied constant is allowed to depend on $A$. But then it follows that  there exists a constant $C_A>0$ depending only on $A$ such that 
$$
\mathbb{E}_r(T;S_A)\leq  C_A\frac{\rho(r)}{|r|^2},
$$
where
$$
\rho(r)=\#\left\{(u,v,w)\in (\OK/r\OK)^3: u^3+v^3+w^3\equiv 0 \bmod{r}\right\}.
$$
It now follows that 
$$
\mathbb{E}_1(T;\OK\setminus S_A)\geq |r|^{-1}
\mathbb{E}_r(T;\OK\setminus S_A)\geq |r|^{-1}\left(1- C_A\frac{\rho(r)}{|r|^2}\right),
$$
whence
$$
\mathbb{E}_1(T;S_A)\leq 1-|r|^{-1}\left(1- C_A\frac{\rho(r)}{|r|^2}\right).
$$
The proof of Theorem \ref{THM:fixed-A}  therefore reduces to finding 
 an element $r\in \Omon$, depending only on $A$,  such that 
 $$
\frac{ \rho(r)}{|r|^2}<C_A^{-1}.
 $$

For any finite field $\FF_m$ of characteristic $p$, 
we introduce the  {\em cubic Gauss sum} 
$$
g(\chi)=\sum_{u\in \FF_m} \chi(u) e_p(T_{\FF_m/\FF_p}(u)),
$$
where $\chi$ is any  non-trivial multiplicative cubic character of $\FF_m$
and where $T_{\FF_m/\FF_p}: \FF_m\to \FF_p$ is the trace map.
This has absolute value $\sqrt{m}$ and we define the normalised cubic Gauss sum
$\widetilde{g}(\chi)=g(\chi)/\sqrt{m}$, over $\FF_m$.
 
 Let $\varpi\in \Omon$ be a prime such that 
 $|\varpi|\equiv 1\bmod{3}$
and let $\FF_\varpi$ be the finite field that is isomorphic to $\OK/\varpi\OK$, with cardinality 
$|\varpi|$. 
It follows from Theorem 2 in \cite{IR}*{\S~10.3} that 
$$
\frac{\rho(\varpi)-1}{|\varpi|-1}
=|\varpi|+1+2|\varpi|^{-1}\Re g(\chi_\varpi)^3,
$$
where $\chi_\varpi$ is any non-trivial cubic character of $\FF_\varpi$.
But then 
\begin{equation}\label{eq:tap}
\frac{\rho(\varpi)}{|\varpi|^2}=
1+\frac{c_\varpi}{|\varpi|}-\frac{c_\varpi}{|\varpi|^2},
\end{equation}
where $c_\varpi=2|\varpi|^{-1}\Re g(\chi_\varpi)^3\in \ZZ$.  

The 
normalised cubic Gauss sum
$\widetilde{g}(\chi_\varpi)=g(\chi_\varpi)/|\varpi|^{1/2}$ is a complex number with absolute value $1$ and we may write
 $$
 c_\varpi=2|\varpi|^{-1}\Re g(\chi_\varpi)^3=2|\varpi|^{1/2}\Re \widetilde g(\chi_\varpi)^3, 
 $$
 so that $|c_\varpi|\leq 2|\varpi|^{1/2}$.
 We would like to prove that there are infinitely many primes 
 $\varpi\in \Omon$ with 
 $|\varpi|\equiv 1\bmod{3}$ for which $c_\varpi<0$. 
This would follow from a function field version of the result by Heath-Brown and Patterson \cite[Theorem~2]{HBP}, but such a result has yet to be worked out in the literature. Fortunately, since we work with characters over the constant field, an elementary approach is available to us. 

Pick an integer $d\geq 1$.
The  prime number theorem for function fields implies that there are 
 $(2d)^{-1}q^{2d}+O(q^d)$ primes 
 $\varpi\in \Omon$ such that $\deg(\varpi)=2d$.
 Moreover, any such prime will satisfy 
 $|\varpi|\equiv 1\bmod{3}$, since $2d$ is even. 
For any such prime 
$\varpi$, we may assume that our 
cubic character $\chi_\varpi$ on $\FF_\varpi\subset \overline{\FF}_q$ takes the shape $\chi\circ N_{\FF_{\varpi}/\FF_{q^2}}$, where $\chi$ is a fixed cubic character on $\FF_{q^2}\subset \overline{\FF}_q$ and 
$N_{\FF_{\varpi}/\FF_{q^2}}:\FF_\varpi\rightarrow \FF_{q^2}$ is the norm map.
It then follows from the Hasse--Davenport relation that 
$
\widetilde g(\chi_\varpi)=-(-\widetilde g(\chi))^d,
$
where $\widetilde g(\chi)$ is the normalised Gauss sum on $\FF_{q^2}$, 
whence
\begin{equation}\label{eq:wire}
 c_\varpi=-2|\varpi|^{1/2} \Re \left(-\widetilde g(\chi)^3\right)^d. 
  \end{equation}
We proceed to prove the following result.

\begin{lemma}\label{lem:infinitely-d}
There exist infinitely many integers $d\ge 0$ such that
$$
\tfrac{9}{10}\leq \Re \left(-\widetilde g(\chi)^3\right)^d \leq 1.
$$
\end{lemma}

\begin{proof}
Let $\alpha\in [-\frac{1}{2},\frac{1}{2}]$ be such that $-\widetilde g(\chi)^3=e(\alpha)$.
Suppose first that $\alpha=a/b\in \QQ$.
Then any integer $d\ge 0$ that is divisible by $b$
satisfies $e(d\alpha)=1$ and thus suffices for the statement of the lemma.

Suppose next that $\alpha\not\in \QQ$. 
The upper bound is trivial for any $d\in \ZZ_{\geq 0}$ and so we focus on  the lower bound.
It follows from Kronecker's approximation theorem that 
there exist infinitely many integers $M\in \NN$ such that $\|M\alpha\|\le \frac{1}{99}$, where $\|\eta\|$ denotes the distance to the nearest integer from any $\eta\in \RR$.
For any such $M$, taking $d=M$ gives $\cos (2\pi \cdot d\alpha)=\cos(2\pi \gamma)$, for some $\gamma\in [-\frac{1}{99},\frac{1}{99}]$. But then it follows that $\cos(2\pi \gamma)\geq \frac{9}{10}$, meaning 
that the statement of the lemma holds with $d=M$.
\end{proof}

Choose $d$ in Lemma~\ref{lem:infinitely-d}
such that $d$ is large enough in terms of $q$.
 Then  we can find $m\geq q^{2d}/(4d)$ primes 
  $\varpi_1,\dots,\varpi_m\in \Omon$ satisfying $|\varpi_i|\equiv 1\bmod{3}$ and  
$\deg(\varpi_i)=2d$, 
for $1\leq i\leq m$, 
and for which 
\begin{align*}
c_{\varpi_i}\left(
\frac{1}{|\varpi_i|}-\frac{1}{|\varpi_i|^2}\right)
&\leq -\frac{9}{5}\left(\frac{1}{|\varpi_i|^{1/2}}-
\frac{1}{|\varpi_i|^{3/2}}\right)\\
&\leq -\frac{9}{10} \cdot \frac{1}{q^{d}},
\end{align*}
in the light of \eqref{eq:wire}. 
Put  $r_d=\varpi_1\cdots \varpi_m\in \Omon$.  
But then it follows from \eqref{eq:tap} and 
the Chinese remainder theorem that 
$$
\frac{\rho(r_d)}{|r_d|^2}= \prod_{1\leq i\leq m} \frac{\rho(\varpi_i)}{|\varpi_i|^2}
\leq \prod_{1\leq i\leq m}\left(
1-\frac{9}{10} \cdot \frac{1}{q^{d}}\right)
\leq \prod_{1\leq i\leq m}\exp\left(
-\frac{9}{10} \cdot \frac{1}{q^{d}}\right),
$$
since $1-\xi\leq \exp(-\xi)$ for any $\xi\in \RR$.
On recalling that 
$m\geq q^{2d}/(4d)$, we deduce that 
$$
\frac{\rho(r_d)}{|r_d|^2}\leq  
\exp\left(
-\frac{9}{10} \cdot \frac{q^{d}}{4d}\right).
$$
This is  strictly less than $C_A^{-1}$ on taking $d$ to be sufficiently large in terms of $A$, which thereby  completes the proof of Theorem \ref{THM:fixed-A}.

\section{Local densities}

Suppose $\cha(\FF_q) \ne 3$.
In \S~\ref{Sec: VarAnalysis}, we will study a suitable arithmetic variance of the global counting function
\begin{equation*}
r_A(k)
\defeq \#\left\{(x,y,z)\in \OK^3: \begin{array}{l}
x^3+y^3+z^3=k, \\
\deg{x}=\deg{y} \le \tfrac{\deg{k}}{3}+A,\\
\deg y-1 \leq \deg z \leq \deg y
\end{array}\right\}
\end{equation*}
over $q$-adic intervals $\abs{k} = \hat B$ as $B\to \infty$, for fixed $A\ge 0$.
Write $B=3d+\alpha$, where $d\in \ZZ$ and $\alpha\in \{0,1,2\}$,
and define $\tilde{A}=\lfloor \alpha/3+A\rfloor$.
Let us also set $P=t^d$. To prove Theorem~\ref{THM:main-intro} we may take $A$ as large as we wish, and we shall therefore assume $A\geq 2$ in all that follows.

Before proceeding to the variance analysis, we will need to introduce the relevant local densities and show that they  are not too small on average. We begin by  defining the  key weight function in 
our analysis.

\begin{definition}
\label{DEFN:A-skew-weights}
Let $A\geq 0$ and $\alpha\in \{0,1,2\}$ be given. Over $K_\infty$ we define 
\begin{equation*}
\begin{split}
\nu_{A,\alpha}(x,y,z) &\defeq \bm{1}_{\abs{x^3+y^3+z^3} = q^\alpha}
\bm{1}_{1\le \abs{x}=\abs{y}\le q^{\alpha/3+A}}\bm{1}_{|y|q^{-1}\leq |z|\leq |y|}, \\
w_{A,\alpha}(x_1,\dots,x_6) &\defeq \nu_{A,\alpha}(x_1,x_2,x_3)\nu_{A,\alpha}(x_4,x_5,x_6).
\end{split}
\end{equation*}
We will typically write $\nu=\nu_{A,\alpha}$ and $w=w_{A,\alpha}$.
\end{definition}
\begin{remark}
    A few comments are needed to explain our choice of weight function. Firstly, it is not strictly necessary to state $1\leq |x|, |y|$, as this follows from $|x^3+y^3+z^3|=q^\alpha$ and $|z|\leq |y|=|x|$ already. Secondly, it is crucial for our argument that $x, y, z$ are of roughly the same size. This guarantees that there are not too many points lying on linear subspaces on the Fermat hypersurface $\sum_{i=1}^6x_i^3=0$, as seen in Lemma~\ref{Le: NumberPtsLinSpace}. Thirdly, we cannot simply take $|x|=|y|=|z|$, but need to allow for one variable to  be smaller than the others. Indeed, if $|k|=q^{3d+\alpha}$ with $\alpha \neq 0$, then $x^3+y^3+z^3=k$ can only hold if the leading coefficient of $x^3+y^3+z^3$ cancels. When $|x|=|y|=|z|$, then this is only possible if there exists a solution $(x_0,y_0,z_0)\in (\FF_q^\times)^3$ such that $x_0^3+y_0^3+z_0^3=0$. But such solutions do not exist for $q= 2,4,7,13,16$,
as was implicitly observed by Serre and Vaserstein \cite{vaserstein1991sums}*{p.~350}.
\end{remark}
If $|k|=\hat B$, as above, with $B=3d+\alpha$, then we may write
$$
r_A(k)=\sum_{\substack{(x,y,z)\in \OK^3\\ x^3+y^3+z^3=k}}
\nu\left(\frac{x}{P},\frac{y}{P},\frac{z}{P}\right),
$$
where we recall that $P=t^d$.
It is now time to introduce the relevant local  densities. 

The analogue of the  real density is given by 
\begin{equation}\label{eq:real}
\sigma_{\infty,A}(k)=\int_{|\theta|\leq q^{4\tilde{A}}}\psi(-\theta kP^{-3})\int_{K_\infty^3}\nu(x,y,z)\psi(\theta (x^3+y^3+z^3))\dd x \dd y \dd z \dd \theta.
\end{equation}
Next, let  $M\ge 0$ be a parameter to be chosen in due course and define
\begin{equation}\label{Eq: DefiN}
N = \prod_{\substack{\varpi\in \Omon \textnormal{ prime}\\\deg\varpi\le M}} \varpi^{\floor{M/\deg\varpi}}.
\end{equation}
The relevant local densities at the finite places are conveniently bundled together in the expression
\begin{equation}\label{eq:def-rho}
\tilde{\rho}(N,k)=|N|^{-2}\#\{(x,y,z)\in (\OK/N\OK)^3\colon x^3+y^3+z^3\equiv k\bmod{N}\}.
\end{equation}
Then, in \S~\ref{Sec: VarAnalysis} 
we shall compare $r_A(k)$ on average over $k$ with the local counting function
\begin{equation}\label{eq:local}
l_A(k;M)
\defeq \sigma_{\infty,A}(k)\tilde{\rho}(N,k).
\end{equation}

In the remainder of this section we prove some preliminary bounds on the typical sizes of 
$\sigma_{\infty,A}(k)$ and $\tilde{\rho}(N,k)$.
Given $k\in K_\infty$, it will be convenient to introduce the notation $F_k(x,y,z)=x^3+y^3+z^3-k$ and to set $F(\bm{x})=F_0(x_1,x_2,x_3)+F_0(x_4,x_5,x_6)$ for $\bm{x}=(x_1,\dots, x_6)\in K_\infty^6$. We begin by relating the real density to a point count. 
\begin{lemma}\label{Le: RealDens=PointCt}
    Given $k\in \OK$ such that $|k|=|P|^3q^{\alpha}$ with $\alpha\in\{0,1,2\}$, let $m=P^{-3}t^{-3\tilde{A}}k$. Writing $s=t^{-1}$, we have 
    \[
    \sigma_{\infty,A}(k)=\frac{1}{q^{14\tilde{A}+2}}\#\left\{\bm{x}\in\left(\frac{\FF_q[s]}{(s^{7\tilde{A}+2})}\right)^3\colon \begin{array}{l}
        F_0(\bm{x})\equiv m\bmod s^{7\tilde{A}+2},\\ v_s(x_1)=v_s(x_2)\leq \tilde{A},\\
        v_s(x_2)\leq v_s(x_3)\leq v_s(x_2)+1
    \end{array}\right\}.
    \]
\end{lemma}
\begin{proof}
    Interchanging the order of integration and executing the integral over $\theta$ in the definition of $\sigma_{\infty,A}(k)$ yields
    \begin{equation}\label{Eq: aaaaaaa}
    \sigma_{\infty,A}(k)= q^{4\tilde{A}+1}\vol\{(x,y,z)\in\supp \nu\colon |F_0(x,y,z)-kP^{-3}|<q^{-4\tilde{A}-1}\}.
    \end{equation}
    Suppose that $\bm{y}\in K^3_\infty$ satisfies $|\bm{y}|<q^{-6\tilde{A}-1}$ and let $t^{\tilde{A}}\bm{x}\in \supp \nu$. 
    Let us define ${(x,y,z)=t^{\tilde{A}}\bm{x}+\bm{y}}$. Then using Taylor expansion, and recalling that $\tilde{A}\geq \alpha$ one sees that $|F_0(x,y,z)-kP^{-3}|<q^{-4\tilde{A}-1}$ holds if and only if 
    \begin{equation}\label{Eq: SizeFRealDens}
    |F_0(\bm{x})-t^{-3d-3\tilde{A}}k|<q^{-7\tilde{A}-1}.    
    \end{equation}
    In addition, if this holds, then $|F_0(x,y,z)|=q^\alpha$ is automatically satisfied and hence $(x,y,z)\in \supp\nu$ if and only if 
    \begin{equation}\label{Eq: SizexRealdens}
    q^{-\tilde{A}}\leq |x_1|=|x_2| \quad\text{and}\quad |x_2|q^{-1}\leq |x_3|\leq |x_2|    
    \end{equation}
    for $\bm{x}=(x_1,x_2,x_3)$. It follows that the set on the right hand side of \eqref{Eq: aaaaaaa} is invariant under translation by elements of $t^{-6\tilde{A}-1}\TT^3$ so that the volume in the right hand side of \eqref{Eq: aaaaaaa} factors through the quotient group $(t^{\tilde{A}+1}\TT/t^{-6\tilde{A}-1}\TT)^3$. 
    
    Let $s=t^{-1}$. Then for any $l\in \ZZ_{\geq 0}$ we have a natural identification $s^{-1}\TT/s^l\TT\simeq \FF_q[s]/(s^{l+1})$ in the category of rings. Under this identification, for any $x\in s^{-1}\TT/s^l\TT$ we have $|x|=q^{-v_s(x)}$, so that \eqref{Eq: SizeFRealDens} and \eqref{Eq: SizexRealdens} hold if and only 
    \[
    F_0(\bm{x})\equiv s^{3(d+\tilde{A})}k\bmod{ s^{7\tilde{A}+2}}
    \]
    and 
    \[
    v_s(x_1)=v_s(x_2)\leq \tilde{A},\quad v_s(x_2)\leq v_s(x_3)\leq v_s(x_2)+1.
    \]
    Note that this makes sense as $s^{3(d+\tilde{A})}k\in \FF_q[s]$. 
    The desired result now follows, since $\vol(\bm{y}\in K_\infty^3\colon |\bm{y}|<q^{-6\tilde{A}-1}\}=q^{-3(6\tilde{A}+1)}$.
    \end{proof}
\begin{proposition}\label{Prop: RealDensitylarge}
    Uniformly over $k\in \OK$, we have 
    $\sigma_{\infty,A}(k) \gg A$.
\end{proposition}
\begin{proof}
    For any integers $1\leq b +1\leq c\leq 7\tilde{A}+2$ with $c\geq 1$, let 
    \[
    n_b(c)=\#\{\bm{x}\in (\FF_q[s]/s^c)^3\colon F_0(\bm{x})\equiv m \bmod s^c,\; v_s(x_1)=v_s(x_2)=b,\; v_s(x_3)=b+1\}.
    \]
    It is then clear from Lemma~\ref{Le: RealDens=PointCt} that 
    \begin{equation}\label{Eq: Lowerboundsigmakinfty}
        \sigma_{\infty,A}(k)\geq q^{-14\tilde{A}-2}\sum_{b=0}^{\tilde{A}-1}n_b(7\tilde{A}+2).
    \end{equation}
    Let us begin with the contribution from $b=0$. We have $v_s(m)=v_s(s^{3d+3\tilde{A}}k)=3\tilde{A}>0$. In particular, $F_0(\bm{x})\equiv m \bmod s$ holds for $\bm{x}$ such that $v_s(x_1)=v_s(x_2)=0$ and $v_s(x_3)=1$ if and only if $x_1^3+x_2^3\equiv 0 \bmod s$. One such solution $\bm{a}$ is provided by taking $a_1=1$ and $a_2=-1$. As $3\nmid \cha(\FF_q)$, this solution is smooth and a Hensel lifting argument shows that there exists a solution $\bm{x}\in \FF_q[s]/(s^2)$ with $\bm{x}\equiv \bm{a}\bmod s$ and $x_3\not\equiv 0\bmod s^2$. Using Hensel lifting again shows that any such $\bm{x}$ lifts to $q^{14\tilde{A}}$ solutions modulo $s^{7\tilde{A}+2}$, so that 
    \begin{equation}\label{Eq: lowerboundn0}
    n_0(7\tilde{A}+2)\geq q^{14\tilde{A}}. 
    \end{equation}
    Now let us assume that $b\geq 1$. Then any $\bm{x}$ counted by $n_b(7\tilde{A}+2)$ can be written as $\bm{x}=s^b\bm{x}'$, where $\bm{x}'\in \FF_q[s]/(s^{7\tilde{A}+2-b})$ is such that $v_s(x_1')=v_s(x_2')=0$ and $v_s(x_3')=1$. Let $m_b=ms^{-3b}$ and note that $v_s(m_b)=3\tilde{A}-3b \geq 3$. We thus see that $\bm{x}$ is counted by $n_b(7\tilde{A}+2)$ if and only if $F_0(\bm{x}')\equiv m_b\bmod s^{7\tilde{A}+2-3b}$. This only depends on $\bm{x}'$ modulo $s^{7\tilde{A}+2-3b}$ and hence 
    \[
    n_b(7\tilde{A}+2)\geq q^{6b}n_0'(7\tilde{A}+2-3b),
    \]
    where $n_0'$ is defined as $n_0$ with $m$ replaced by $m_b$. As $v_s(m_b)\geq 1$, the argument leading to \eqref{Eq: lowerboundn0} also applies to $n_0'$ and yields 
    \begin{equation}\label{Eq: lowerboundnb}
        n_b(7\tilde{A}+2)\geq q^{6b}n_0'(7\tilde{A}+2-3b)\geq q^{14\tilde{A}}.
    \end{equation}
    Inserting \eqref{Eq: lowerboundnb} into \eqref{Eq: Lowerboundsigmakinfty} gives
    \begin{align*}
    \sigma_{\infty,A}(k)&\geq q^{-14\tilde{A} -2}\sum_{b=0}^{\tilde{A}-1}q^{14\tilde{A}}
    = \frac{\tilde{A}}{q^2}. 
    \end{align*}
    Since $\tilde{A} > A-1 \gg A$, and $q$ is fixed, the proposition follows.
\end{proof}

Finally, for the local densities 
\eqref{eq:def-rho}
at the finite places we require the following result,
showing that on average they are not too small.
\begin{proposition}\label{Prop: LocDenAvg}
    Let $N\in \Omon$. Then 
    \[
    |N|^{-1}\sum_{|k|<|N|}\frac{1}{\tilde{\rho}(N,k)}\ll 1.
    \]
\end{proposition}
\begin{proof}
The quantity $\tilde{\rho}(N,k)$ is a multiplicative function of $N$ by the Chinese remainder theorem, so that 
$$
|N|^{-1}\sum_{|k|<|N|}
\frac{1}{\tilde{\rho}(N,k)}=
|N|^{-1}\sum_{|k|<|N|}
\prod_{\varpi^e\| N}\frac{1}{\tilde{\rho}(\varpi^e,k)}.
$$
Since the value of 
$\tilde{\rho}(\varpi^e,k)$ only depends on $k$ modulo $\varpi^e$, we may further write
\begin{equation}\label{eq:lemon}
|N|^{-1}
\sum_{|k|<|N|}\frac{1}{\tilde{\rho}(N,k)}=
\prod_{\varpi^e\| N}
A(\varpi^e),
\end{equation}
where
$$
A(\varpi^e)=\frac{1}{|\varpi^e|}
\sum_{|k|<|\varpi^e|}
\frac{1}{\tilde{\rho}(\varpi^e,k)}.
$$
We need to get a sufficiently sharp upper bound for $A(\varpi^e)$.

Let $\tilde{\rho}^*(\varpi^e,k)$ be defined as for $\tilde{\rho}(\varpi^e,k)$, but with the extra constraint that 
$(x,y,z)$ should be coprime to $\varpi$. Then clearly 
$$
\tilde{\rho}(\varpi^e,k)\geq \tilde{\rho}^*(\varpi^e,k)=\tilde{\rho}^*(\varpi,k),
$$
by Hensel's lemma and the fact that the characteristic is not  $3$.
Thus it follows that 
\begin{equation}\label{eq:lime}
\begin{split}
A(\varpi^e)&\leq 
\frac{1}{|\varpi|}
\sum_{|k|<|\varpi|}
\frac{1}{\tilde{\rho}^*(\varpi,k)}\\
&=
\frac{1}{|\varpi|}
\frac{1}{\tilde{\rho}^*(\varpi,0)}+
\frac{1}{|\varpi|}
\sum_{0<|k|<|\varpi|}
\frac{1}{\tilde{\rho}^*(\varpi,k)},
\end{split}
\end{equation}
According to \eqref{eq:tap}, we have 
$$
\tilde{\rho}^*(\varpi,0)=1+\frac{c_{\varpi}}{|\varpi|}-\frac{c_\varpi+1}{|\varpi|^2},
$$
where $c_\varpi\in \ZZ$ satisfies $|c_\varpi|\leq 2\sqrt{|\varpi|}$.
Hence there exists an absolute constant $C_1>0$ such that 
$$
\frac{1}{\tilde{\rho}^*(\varpi,0)}\leq 1+ \frac{C_1}{\sqrt{|\varpi|}}.
$$

We proceed to an analysis of 
$\tilde{\rho}^*(\varpi,k)$ when $\varpi\nmid k$. Identifying $k$ with the image of its reduction modulo $\varpi$, it is  clear that 
$$
\tilde{\rho}^*(\varpi,k)=\tilde{\rho}(\varpi,k)=
\frac{1}{r^2}\#\{(x,y,z)\in \FF_r^3: x^3+y^3+z^3=k\},
$$
where $r=|\varpi|$ and $\FF_r\cong \OK/\varpi \OK$.
Let us write 
 $\nu(r)=\#\{(x,y,z)\in \FF_r^3: x^3+y^3+z^3=k\}$.
 We may now appeal to the formulae in \cite[Chap.~8]{IR} to calculate this quantity. 
If $r\equiv 2\bmod{3}$ then $\nu(r)=r^2$.
If $r\equiv 1\bmod{3}$,  we let $\chi:\FF_r^*\to \CC$ be a non-trivial character of order $3$. Then one may check that 
$$\nu(r)=r^2+3\left(\chi(k)+\bar\chi(k)\right)r-2\Re J(\chi,\chi),
$$
where $J(\chi,\chi)$ is a Jacobi sum and satisfies $|J(\chi,\chi)|= \sqrt{r}$.
We deduce that 
$$
\frac{1}{\tilde{\rho}^*(\varpi,k)}\leq 
1-\frac{3(\chi(k)+\bar\chi(k))}{|\varpi|} +\frac{C_2}{|\varpi|^{3/2}},
$$
for a suitable absolute  constant $C_2>0$.

Bringing these estimates together in \eqref{eq:lime}, we  obtain
\begin{align*}
A(\varpi^e)
&\leq \frac{1}{|\varpi|}
\left(1+ \frac{C_1}{\sqrt{|\varpi|}}\right)
+
\frac{1}{|\varpi|}
\sum_{0<|k|<|\varpi|}1\\
&\leq 1 + \frac{C_1}{|\varpi|^{3/2}},
\end{align*}
if $|\varpi|\equiv 2\bmod{3}$. On the other hand, if 
$|\varpi|\equiv 1\bmod{3}$, then
\begin{align*}
A(\varpi^e)
&\leq \frac{1}{|\varpi|}
\left(1+ \frac{C_1}{\sqrt{|\varpi|}}\right)
+
\frac{1}{|\varpi|}
\sum_{0<|k|<|\varpi|}
\left(
1-\frac{3(\chi(k)+\bar\chi(k))}{|\varpi|} +\frac{C_2}{|\varpi|^{3/2}}\right)\\
&\leq 1 -\frac{3}{|\varpi|^2}
\sum_{0<|k|<|\varpi|}
(\chi(k)+\bar\chi(k)) 
+\frac{C_1+C_2}{|\varpi|^{3/2}}.
\end{align*}
But the sum  over $k$ vanishes, since $\chi,\bar\chi$ are non-trivial characters modulo $\varpi$ and the image of $k$ runs over all of  $\FF_r^*$ on reduction modulo $\varpi$.

Returning to \eqref{eq:lemon}, we conclude that 
$$
|N|^{-1}
\sum_{|k|<|N|}\frac{1}{\tilde{\rho}(N,k)}\leq 
\prod_{\varpi \mid N}
\left(
1 + \frac{C_1+C_2}{|\varpi|^{3/2}}\right)\ll 1,
$$
as required.
\end{proof}

\section{Variance analysis for density 1}\label{Sec: VarAnalysis}

For the rest of the paper, assume $\cha(\FF_q) > 3$.
Having carried out our study of local densities, in this section we turn to a variance analysis of $r_A(k)$ and provide a proof of Theorem~\ref{THM:main-intro}, assuming Theorem~\ref{Thm: Manin}, whose proof will be completed in \S~\ref{Sec: CentreDual}. 

\newcommand{\Var}{\operatorname{Var}}

Recall the definition 
 \eqref{eq:local}
of $l_A(k;M)$.
We are interested in the arithmetic variance
\begin{equation*}
\Var_A(B;M) = \sum_{\substack{k\in \OK\\ \abs{k}=\hat B} }(r_A(k) - l_A(k;M))^2,
\end{equation*}
ranging over all $k\in \OK$ with $\deg{k}=B$.
We will analyse this roughly as in \cite{wang2022thesis}*{Chapter~2}, which is in turn based on arguments of \cite{diaconu2019admissible}. A crucial role in our  analysis is played by the counting function 
\begin{equation}\label{eq:fels}
    N_w(P)\coloneqq \sum_{\substack{\bm{x}\in \OK^6 \\ F(\bm{x})=0}}w(\bm{x}/P),
\end{equation}
where $w$ is the weight function constructed in 
Definition \ref{DEFN:A-skew-weights}.

We then define 
\begin{align}
    \Sing &=\sum_{\substack{r\in \OK^+}}|r|^{-6}S_r(\bm{0}),\label{Eq: SingSer}
\end{align}
to be the {\em singular series} associated to our counting function $N_w(P)$, 
where 
\begin{equation}\label{eq:sun}
S_r(\bm{0})= \sum_{\substack{|a|<|r|\\ \gcd(a,r)=1}}
\sum_{|\bm{x}|<|r|}\psi\left(\frac{aF(\bm{x})}{r}\right).
    \end{equation}
Moreover, the {\em singular integral} is defined to be
\begin{align}
        \sigma_{\infty} &=\int_{|\theta|\leq q^{4\tilde{A}}}\int_{K_\infty^6}
        w(\bm{x})\psi(\theta F(\bm{x}))\dd\bm{x}\dd\theta.\label{Eq: singInt}
\end{align}

We have 
\begin{align*}
    \Var_A(B;M) & = \sum_{|k|=\hat{B}}r_A(k)^2 - 2 \sum_{|k|=\hat{B}}r_A(k)l_A(k;M) + \sum_{|k|=\hat{B}}l_A(k;M)^2 \\
    &= \Sigma_1 -2 \Sigma_2 + \Sigma_3,
\end{align*}
say. We will analyse each term $\Sigma_i$ individually, for $i=1,2,3$, with the main bulk of our work being concerned with producing an asymptotic for $\Sigma_1$. Before commencing our analysis, we will need two auxiliary results. 
\begin{lemma}\label{Le: Avg.singser.squared}
    Let $N\in \OK$, $B=3d+\alpha\geq 0$ and suppose that $|N|<|P|^3q^{-4\tilde{A}}$. Then for any $a\in \OK$ with $|a|<|N|$, we have 
    \[ \mathcal{Q} \defeq \sum_{\substack{|k|=\hat{B} \\ k\equiv a \bmod N}}\sigma_{\infty,A}(k)^2 = \sigma_\infty \frac{|P|^3}{|N|},
    \]
    where $\sigma_{\infty,A}(k)$ is given by \eqref{eq:real}.
\end{lemma}
\begin{proof}
    After writing $k=Nb+a$, with $|b|=\hat{B}/|N|$ and $|a|<|N|$, it follows that 
    \begin{align*}
    \mathcal{Q} =&\int_{|\theta_1|, |\theta_2|\leq q^{4\tilde{A}}}\int_{K_\infty^6}w(\bm{x})\psi(\theta_1 F_0(\bm{x}_1)+\theta_2 F_0(\bm{x}_2)-a(\theta_1+\theta_2)P^{-3}) \\
    &\times \sum_{|b|=\hat{B}/|N|}\psi(-(\theta_1+\theta_2)bNP^{-3})\dd\bm{x}_1\dd \bm{x}_2\dd\theta_1\dd\theta_2.
    \end{align*}  
    The assumption $|N|<|P|^3q^{-4\tilde{A}}$ implies 
    \[
    |a(\theta_1+\theta_2)P^{-3}| < |N|q^{4\tilde{A}}|P|^{-3}<1,
    \]
    so that $\psi(a(\theta_1+\theta_2)P^{-3})=1$. Moreover, it follows from orthogonality of additive characters that 
    \[
    \sum_{|b|=\hat{B}/|N|}\psi(-(\theta_1+\theta_2)bNP^{-3}) = \frac{\hat{B}}{|N|}\left(q\mathbf{1}_{|\norm{(\theta_1+\theta_2)NP^{-3}}|<\frac{|N|}{q\hat{B}}}-\mathbf{1}_{|\norm{(\theta_1+\theta_2)NP^{-3}}|<\frac{|N|}{\hat{B}}}\right).
    \]
    We again have $|(\theta_1+\theta_2)NP^{-3}|<1$, so that $||(\theta_1+\theta_2)NP^{-3}||=|(\theta_1+\theta_2)NP^{-3}|$. After making the change of variables $\theta_2=-\theta_1+\gamma$ and $\theta=\theta_1$, it follows that 
     \begin{align*}
    \mathcal{Q}
     = \frac{\hat{B}}{|N|} \int_{K_\infty^6}\int_{|\theta|, |\gamma-\theta|\leq q^{4\tilde{A}}}
     w(\bm{x})\psi(\theta F(\bm{x})-\gamma F_0(\bm{x}_2))  \left(q\mathbf{1}_{|\gamma|<q^{-\alpha -1}}-\mathbf{1}_{|\gamma|<q^{-\alpha}}\right)\dd \gamma \dd \theta\dd \bm{x}, 
    \end{align*}
    where we applied the change of variables $\bm{x}=(\bm{x}_1, -\bm{x}_2)$. If $w(\bm{x})\neq 0$, then $|F_0(\bm{x}_2)|=q^\alpha$. Moreover, if $|\gamma|<q^{-\alpha}$ and $|\theta|\leq q^{4\tilde{A}}$, then $|\gamma-\theta|\leq q^{4\tilde{A}}$ holds automatically. This implies 
    \begin{align*}
        \int_{|\gamma|<q^{-\alpha}}\psi(-\gamma F_0(\bm{x}_2))\dd\gamma = q^{-\alpha}\mathbf{1}_{|F_0(\bm{x}_2)|<q^\alpha}=0. 
    \end{align*}
In addition, if $|\gamma|<q^{-1-\alpha}$, then $\psi(-\gamma F_0(\bm{x}_2))=1$ since $|F_0(\bm{x}_2)|=q^\alpha$ for $\bm{x}_2\in \supp \nu$. We therefore obtain 
\begin{align*}
    \mathcal{Q} &= \frac{q\hat{B}}{|N|}\int_{K_\infty^6}\int_{|\theta|\leq q^{4\tilde{A}}}w(\bm{x})\psi(\theta F(\bm{x}))\dd\theta \dd\bm{x}\vol(\{|\gamma|<q^{-1-\alpha}\})\\
    &= \frac{|P|^3}{|N|}\sigma_\infty    
\end{align*}
since $\hat{B}=|P|^3q^\alpha$.
\end{proof}
Let $\bm{y}\in K_\infty^3$ be such that $|\bm{y}|\leq q^{\tilde{A}}$ and $\bm{z}\in \TT^3$. If we set  $\bm{x}=\bm{y}+t^{\alpha -2\tilde A}\bm{z}$,
then a straightforward computation shows that $|F(\bm{x})|=q^\alpha$ if and only if $|F(\bm{y})|=q^\alpha$. Similarly $\nu(\bm{x})\neq 0$ holds if and only if $\nu(\bm{y})\neq 0$. In particular, $\nu$ is invariant under translation by $t^{\alpha -2\tilde{A}}\TT^3$. If we define 
\begin{equation}\label{eq:RA}
    R_{A,\alpha}=\left\{\bm{y}\in (t^{\tilde{A}+1}\TT/t^{\alpha -2\tilde{A}}\TT)^3\colon \begin{array}{l}
|F_0(\bm{y})|=q^\alpha, \\
1\leq |y_1|=|y_2|\leq q^{\tilde{A}}, \\
|y_2|q^{-1}\le |y_3|\leq |y_2|\end{array}\right\},
\end{equation}
we thus have the well-defined identity
\begin{equation}\label{Eq: DecompWeight}
    \nu(\bm{x})=\sum_{\bm{y}\in R_{A,\alpha}}\mathbf{1}_{|\bm{x}-\bm{y}|<q^{\alpha -2\tilde A}}.
\end{equation}
\begin{lemma}\label{Le: Avg.nu.singser}
    Let $N\in \OK$, $B=3d+\alpha\geq 0$ and suppose that $|N|<|P|q^{-6\tilde{A}}$. If $\bm{b}\in \OK^3$ is such that $|\bm{b}|<|N|$, then 
    \[
    \sum_{\bm{y}\equiv \bm{b} \bmod N}\nu(\bm{y}/P)\sigma_{\infty,A}(F_0(\bm{y}))= \sigma_\infty \frac{|P|^3}{|N|^3}.
    \]
\end{lemma}
\begin{proof}
    An application of the Poisson summation formula yields 
    \begin{align*}
        \sum_{\bm{y}\equiv \bm{b} \bmod N}&\nu(\bm{y}/P)\sigma_{\infty,A}(F_0(\bm{y})) 
        \\
        &= \frac{|P|^3}{|N|^3}\int_{|\theta|\leq q^{4\tilde{A}}}\int_{K_\infty^3}\nu(\bm{x})\psi(\theta F_0(\bm{x}))\sum_{\bm{v}\in \OK^3}\psi\left(-\frac{\bm{b}\cdot \bm{v}}{N}\right) I(\theta, \bm{v})\dd \bm{x}\dd\theta,
    \end{align*}
        where we have temporarily written
    \[
    I(\theta,\bm{v})=\int_{K_\infty^3}\nu(\bm{y})\psi\left(\theta F_0(\bm{y})+\frac{P \bm{y}\cdot \bm{v}}{N}\right)\dd \bm{y}.
    \]
    The contribution from $\bm{v}=\bm{0}$ is clearly
    \begin{align*}
    \frac{|P|^3}{|N|^3}\int_{|\theta|\leq q^{4\tilde{A}}}\int_{K_\infty^3}\nu(\bm{x})\psi(\theta F_0(\bm{x}))I(\theta,\bm{0})\dd\bm{x}\dd\theta
    &= \frac{|P|^3}{|N|^3}\sigma_\infty.
    \end{align*}   

    Using \eqref{Eq: DecompWeight} to decompose the weight function $\nu$ into smaller boxes and applying a change of variables, we obtain
    \[
    I(\theta,\bm{v})=q^{3(\alpha-2\tilde{A})}\sum_{\bm{y}\in R_{A,\alpha}}\psi\left(\frac{P\bm{y}\cdot\bm{v}}{N}\right)\int_{\TT^3}\psi\left(\theta F_0(\bm{y}+t^{\alpha -2\tilde{A}}\bm{z})+\frac{Pt^{\alpha - 2\tilde{A}} \bm{z}\cdot \bm{v}}{N}\right)\dd\bm{z}.
    \]
    We can now use \cite[Lemma 5.1]{BGW2024positive} to deduce that the inner integral vanishes for any $\bm{y}\in R_{A,\alpha}$, and hence $I(\theta,\bm{v})=0$ unless $|P|q^{\alpha -2\tilde{A}}|N|^{-1}|\bm{v}|\leq \max\{1, |\theta|H_G\}$, where $$G(\bm{z})=F_0(\bm{y}+t^{\alpha -2\tilde{A}}\bm{z}).
    $$ 
    Since $|F_0(\bm{y})|=q^\alpha$, one readily checks that $H_G=q^\alpha$. Consequently, if $\bm{v}\neq \bm{0}$ then $I(\theta,\bm{v})$ can only be non-zero if 
    \[
    |P|q^{\alpha -2\tilde{A}}\leq |N|\max\{1,q^\alpha |\theta|\} \leq |N|q^{\alpha+4\tilde{A}},
    \]
    which is impossible by our assumption on $|N|$.
\end{proof}
For any $r\in \OK$, define 
\[
\tilde{\rho}(r)=|r|^{-5}\#\{\bm{x}\in (\OK/r\OK)^6\colon F(\bm{x})\equiv 0 \bmod r\}.
\]
Recalling the definition \eqref{eq:local} of 
$l_A(k;M)$, 
it is now easy to call upon the  previous two lemmas to evaluate $\Sigma_2$ and $\Sigma_3$. 
Firstly, by the definition of $\nu$ we have 
\begin{align*}
    \Sigma_2 & = \sum_{\substack{\bm{x}\in \OK^3}}\nu(\bm{x}/P)l_A(F_0(\bm{x}); M)\\
    & = \sum_{\bm{b}\bmod N}\tilde{\rho}(N,F_0(\bm{b}))\sum_{\substack{\bm{x}\equiv \bm{b} \bmod N}}\nu(\bm{x}/P)\sigma_{\infty, A}(F_0(\bm{x})).
\end{align*}
On assuming that $|N|<|P|q^{- 6\tilde{A}}$ and appealing to 
    Lemma \ref{Le: Avg.nu.singser}, it now follows that 
 $$
 \Sigma_2= \sigma_\infty |P|^3\sum_{\bm{b}\bmod N} |N|^{-5}\rho(N,F_0(\bm{b})),
$$
where
$\rho(N,k)\defeq |N|^2\tilde{\rho}(N,k)=\#\{(x,y,z)\in (\OK/N\OK)^3\colon x^3+y^3+z^3\equiv k\bmod{N}\}$. 
It readily follows
that 
\begin{align}
\label{Eq: Sigma2}    
\Sigma_2= \sigma_\infty \tilde{\rho}(N)|P|^3.
\end{align}
Similarly, Lemma~\ref{Le: Avg.singser.squared} gives 
\begin{align}
\begin{split}\label{Eq: Sigma3}
    \Sigma_3 & = \sum_{b \bmod N}\tilde{\rho}(N, b)^2 \sum_{\substack{|k|=\hat{B} \\ k\equiv b \bmod N}}\sigma_{\infty,A}(k)^2\\
    & = \sigma_\infty |P|^3 \sum_{b\bmod N}|N|^{-5}\rho(N, b)^2\\
    & = \sigma_\infty \tilde{\rho}(N)|P|^3,        
\end{split}
\end{align}
providing only that $|N|<|P|^3q^{-4\tilde{A}}$. 

Finally, assuming (R2) and Conjecture~\ref{CNJ:RA1} it follows from Theorem~\ref{Thm: Manin} that 
\begin{align}
\begin{split}\label{Eq: Sigma1}
    \Sigma_1 & = \sum_{|k|=\hat{B}}r_A(k)^2 \\
    &= N_w(P) \\
    &= \sigma_\infty \Sing |P|^3 +\sum_{L\in \Upsilon}\sum_{\bm{x}\in L\cap \OK^6}w(\bm{x}/P)+o_w(|P|^3)  ,      
\end{split}
\end{align}
as $|P|\to \infty$. 
Before completing the proof of Theorem~\ref{THM:main-intro}, we need two more auxiliary results.
\begin{lemma}\label{Le: NumberPtsLinSpace}
    For any $L\in \Upsilon$, we have 
    \[
    \sum_{\bm{x}\in L\cap\OK^6}w(\bm{x}/P)\ll A|P|^3.
    \]
\end{lemma}
\begin{proof}
    By applying a suitable $\FF_q$-linear change of variables we may reduce to the case where $L$ is given by
    \begin{enumerate}
        \item $x_1+x_4=x_2+x_5=x_3+x_6=0$, or
        \item $x_1+x_2=x_3+x_4=x_5+x_6=0$. 
    \end{enumerate}
    We begin with (1). Since $x_1,x_2,x_3$ determine $x_4,x_5,x_6$ uniquely, it suffices to count $\bm{x}\in\OK^3$ for which $\nu(\bm{x}/P)\neq 0$. Let $1 \leq b \leq \tilde{A}$ and suppose that $|x_1|=|x_2|=q^{b}|P|$. Writing $N=x_1^3+x_2^3$, we must then have $|x_1|=|x_2|\leq |x_3|q$ and
    \begin{equation}\label{Eq: N+x3=0}
    |N+x_3^3|=q^{\alpha}|P|^3,
    \end{equation}
    if $\nu(\bm{x}/P)\neq 0$. If $N=0$, then there are $O(q^{b}|P|)$ choices for $(x_1,x_2)$, while $|x_3^3|=q^{\alpha}|P|^3$ implies $|x_3|=q^{\alpha/3}|P|$, which in turn implies  that
    $|x_1|=|x_2|\leq q^{\alpha/3+1}|P|$, whence $b\leq \alpha/3+1$. Thus the total contribution from this case is $O(|P|^2)$.

    Now let us assume that $N\neq 0$ and fix pairwise distinct solutions $\beta_1,\beta_2,\beta_3\in K^{\text{sep}}$ of the equation $x^3+N=0$. The absolute value $|\cdot |$ extends uniquely to $K^{\text{sep}}$, so that \eqref{Eq: N+x3=0} gives 
    \[
    |(x_3-\beta_1)(x_3-\beta_2)(x_3-\beta_3)|=q^{\alpha}|P|^3.
    \]
    As $|\beta_i-\beta_j|\gg 1$ for $i\neq j$, we have $|x_3-\beta_i|\gg |x_3|\geq q^{b-1}|P|$ for at least two indices $i\in \{1,2,3\}$. Without loss of generality, assume that it holds for $i=2,3$. We then get 
    \[
    |x_3-\beta_1|\ll q^{\alpha+2-2b}|P|,
    \]
    for which there are $O(q^{\alpha+2-2b}|P|)$ possible $x_3\in \OK$. As there are $O(q^{2b}|P|^2)$ choices for $x_1,x_2\in \OK$, this yields after summing over $b$ an overall contribution of $O(A|P|^3)$, which completes case (1). 

    For (2), observe that if $\nu(\bm{x}/P)\neq 0$ and $x_1+x_2=0$, then we must have $|x_3|=q^{\alpha/3}|P|$ and $|x_1|=|x_2|\leq q^{\alpha/3+1}|P|$. In particular, once $x_3$ is fixed there are $O(q^{\alpha/3}|P|)$ choices for $x_1,x_2$. In addition, $x_3$ determines $x_4$ uniquely and by the symmetry at hand there are $O(q^{\alpha/3}|P|)$ available $x_4,x_5\in \OK$. Therefore, the total contribution is $O(q^{\alpha}|P|^3)$, which is more than sufficient.
\end{proof}
\begin{lemma}\label{Le: ApproxLocDensSingSer}
    Let $N$ be as in \eqref{Eq: DefiN}. Then 
    $
    \tilde{\rho}(N)=\Sing +O(\hat{M}^{-2/3+\varepsilon})$,
where $\mathfrak{S}$ is given by \eqref{Eq: SingSer}.
\end{lemma}
\begin{proof}
    Using orthogonality of characters and collecting terms according to their greatest common divisor, for any prime power $\varpi^k$ we have 
    \begin{align*}
    \tilde{\rho}(\varpi^k)&=   |\varpi|^{-6k}\sum_{a\bmod \varpi^k}\sum_{\bm{x}\bmod \varpi^k}\psi\left(\frac{aF(\bm{x})}{\varpi^k}\right)\\
    &= \sum_{l=0}^k|\varpi|^{-6(k-l)}S_{\varpi^{k-l}}(\bm{0}),
    \end{align*}
    in the notation of \eqref{eq:sun}.
    Thus the Chinese remainder theorem yields 
    \[
    \tilde{\rho}(N)= \sum_{r\mid N}|r|^{-6}S_r(\bm{0}).
    \]
    It follows from the definition of $N$ that any $r\in \Omon$ with $\deg r\leq M$ divides $N$. Therefore
    \begin{align*}
        |\tilde{\rho}(N)-\Sing|&\leq \sum_{|r|> \hat{M}}|r|^{-6}|S_r(\bm{0})|
        \ll \hat{M}^{-2/3+\varepsilon}
    \end{align*}
    by Lemma~9.2 of \cite{BGW2024positive}, noting that 
    $S_r^\natural(\bm{0})=|r|^{-7/2}S_r(\bm{0})$ in this result.
\end{proof}

\begin{proof}[Proof of Theorem~\ref{THM:main-intro}] 
Recall the definition \eqref{def:SA} of $S_A$.
Our goal is to show that 
\[
\liminf_{X\to\infty}\left(\frac{\#\{k\in S_A\colon |k|\leq \hat{X}\}}{q\hat{X}}\right) \to 1, 
\]
as $A\to \infty$.
Fix $A>0$ to be sufficiently large and suppose $0<\delta<1/2$. Let $B\in \NN$ be such that $X/2\leq B\leq X$ and put $B=3d+\alpha$ for $\alpha\in \{0,1,2\}$. Throughout this argument we may take $d$ to be sufficiently large. Let $P=t^d$ and 
choose $M\geq 1$ such that $|N|< |P|q^{-6\tilde{A}}$ in the notation of \eqref{Eq: DefiN}. Then 
\begin{align*}
    \#\{k\in \OK\setminus S_A\colon |k|=\hat{B}\} & \leq \#\left\{k\in \OK\setminus S_A\colon 
    \begin{array}{l}
    |k|=\hat{B},\\
    |r_A(k)-l_A(k,M)|\geq |l_A(k,M)|/2
        \end{array}
    \right\},
\end{align*}
where $l_A(k,M)$ is given by \eqref{eq:local}, since $k\in S_A$ if 
$r_A(k)>0$.
Now it follows from Proposition~\ref{Prop: LocDenAvg} that $\tilde{\rho}(N,k)\geq A^{-1/2+\delta}$ for all but $O(A^{-1/2+\delta}\hat{B})$ elements $k\in \OK$ with $|k|=\hat{B}$. As $\sigma_{\infty,A}(k)\gg A$ always holds, by Proposition~\ref{Prop: RealDensitylarge}, we thus have 
    \begin{align*} 
    \#\{k\in \OK\setminus S_A\colon |k|=\hat{B}\}
& \ll A^{-1/2+\delta}\hat{B} + \frac{\Var_A(B;M)}{A^2 A^{-1+2\delta}}.
\end{align*}
Writing $\Var_A(B;M)= \Sigma_1-2\Sigma_2+\Sigma_3$, as previously, 
we may combine \eqref{Eq: Sigma2}, \eqref{Eq: Sigma3} and \eqref{Eq: Sigma1} with Lemma~\ref{Le: ApproxLocDensSingSer} to get
\[
\Var_A(B;M)=\sum_{L\in \Upsilon}\sum_{\bm{x}\in L\cap\OK^6}w(\bm{x}/P)+o_w(|P|^3)+O(\sigma_\infty |P|^3\hat{M}^{-2/3+\varepsilon}).
\]
But  Lemma~\ref{Le: NumberPtsLinSpace} implies that  $\sum_{L\in \Upsilon}\sum_{\bm{x}\in L\cap \OK^6}w(\bm{x}/P)\ll A|P|^3$. 
We trivially have $\sigma_\infty=O_A(1)$ in 
\eqref{Eq: singInt}. 
Taking $M$ to satisfy $\sigma_\infty \hat{M}^{-2/3+\varepsilon}=o_A(1)$ as $|P|\to \infty$ (which is allowed as $M$ can be chosen to tend to infinity as $|P|\to\infty$) and recalling that $\hat{B}=|P|^3q^{\alpha}$, we are led to the bound
\begin{align*}
    \#\{k\in \OK\setminus S_A\colon |k|=\hat B\} & \ll A^{-1/2+\delta}\hat{B}+\hat{B}A^{-2\delta }+o_{w,A}(|P|^3)
\end{align*}
as $|P|\to \infty$. Choosing $\delta =1/6$, summing over $X/2\leq B\leq X$ and estimating the contribution from $|k|< \hat{X}^{1/2}$ trivially, we conclude that
\[
\liminf_{X\to\infty}\left(\frac{\#\{k\in S_A\colon |k|\leq \hat{X}\}}{q\hat{X}}\right)\geq 1 + O (A^{-1/3}),
\]
as $A\to \infty$. 
\end{proof}

\section{Integral results}

Given 
 $\gamma \in K_\infty$ and  $\bm{w}\in K_\infty^n$, 
this section will be mainly  concerned with the integrals
\[
J_{F,\omega}(\gamma, \bm{w})= \int_{K_\infty^n}\omega(\bm{x})\psi(\gamma F(\bm{x})+\bm{w}\cdot \bm{x})\dd\bm{x}
\]
and 
\[
J^\Gamma_{F,\omega}(\bm{w})=\int_{|\gamma|<\widehat{\Gamma}}J_{F,\omega}(\gamma,\bm{w})\dd\gamma ,
\]
for $\Gamma\in \ZZ$ and suitable weight functions $\omega\colon K_\infty^n\to \CC$.

Fix $A$ and $\alpha$, and let $w=w_{A,\alpha}$ as in Definition~\ref{DEFN:A-skew-weights}.
All implied constants in this section will be allowed to depend on $A$ and $\alpha$.
First, we prove some continuity properties that
enhance \cite{BGW2024positive}*{Lemma~5.6}.

\begin{lemma}\label{LEM:integral-scale-invariance}
Suppose $\lambda_1,\lambda_2 \in K_\infty^\times$ are such that
$\lambda_1/\lambda_2\in 1+\TT$.
Then
\begin{equation*}
J^\Gamma_{F,w}(\lambda_1\bm{w})=J^\Gamma_{F,w}(\lambda_2\bm{w})
\quad\textnormal{for any $\bm{w}\in K_\infty^n$}.
\end{equation*}
Moreover, if $\bm{w}_1, \bm{w}_2\in K_\infty^n$ satisfy $|\bm{w}_1-\bm{w}_2|<q^{-1}(\max_{\bm{x}\in\supp(w)}|\bm{x}|)^{-1}$, then
\begin{equation}
\label{EQN:local-constancy}
J^\Gamma_{F,w}(\bm{w}_1)=J^\Gamma_{F,w}(\bm{w}_2).
\end{equation}
\end{lemma}

\begin{proof}
The first part of the lemma is proven just like \cite{BGW2024positive}*{Lemma~5.6},
using the fact that $w_{A,\alpha}(\lambda^{-1}\bm{y})$ depends only on $\lambda\bmod{1+\TT}$ by Definition~\ref{DEFN:A-skew-weights}.
For the second statement of the lemma, we simply observe that $|\bm{w}_1-\bm{w}_2|<q^{-1}(\max_{\bm{x}\in\supp(w)}|\bm{x}|)^{-1}$ implies that 
\[
\psi(\bm{w}_1\cdot \bm{x})= \psi(\bm{w}_2\cdot\bm{x})
\]
for all $\bm{x}\in\supp(w)$, so that the claim follows from the definition of $J^\Gamma_{F,w}(\bm{w})$. 
\end{proof}

We now extend \cite{BGW2024positive}*{Lemmas~5.5 and~5.8} to the weights $w=w_{A,\alpha}$.
We first place $w$ into the framework of \cite{BGW2024positive}*{\S~5}.
By \eqref{Eq: DecompWeight}, we have
\begin{equation*}
w(t^{\tilde{A}}\bm{x})
= \sum_{\bm{y}=(\bm{u},\bm{v})\in R_{A,\alpha}^2}
\mathbf{1}_{|t^{\tilde{A}}\bm{x}-\bm{y}|<q^{\alpha -2\tilde A}}
= \sum_{\bm{y}\in R_{A,\alpha}^2}
\mathbf{1}_{|\bm{x}-t^{-\tilde{A}}\bm{y}|<q^{\alpha -3\tilde A}}.
\end{equation*}
Moreover, it follows from 
\eqref{eq:RA} that  $\norm{t^{-\tilde{A}}\bm{y}} \le 1$ for all $\bm{y}\in R_{A,\alpha}^2$.
Let $\omega_{\bm{y}}(\bm{x}) \defeq \bm{1}_{\norm{\bm{x}-t^{-\tilde{A}}\bm{y}}<q^{\alpha-3\tilde A}}$.
Then
\begin{equation}
\begin{split}
\label{box-integrals}
J_{F,w}(\gamma, \bm{w})
&= \int_{K_\infty^n} w(t^{\tilde{A}}\bm{x})
\psi(\gamma F(t^{\tilde{A}}\bm{x})+\bm{w}\cdot t^{\tilde{A}}\bm{x})\abs{t^{\tilde{A}}}^n\dd\bm{x} \\
&= q^{n\tilde{A}} \sum_{\bm{y}\in R_{A,\alpha}^2}
J_{F,\omega_{\bm{y}}}(t^{3\tilde{A}}\gamma, t^{\tilde{A}}\bm{w}).
\end{split}
\end{equation}

\begin{lemma}
\label{LEM:w_A-is-nice}
For each $\bm{y}\in R_{A,\alpha}^2$,
the weight function $\omega_{\bm{y}}$ satisfies
\cite{BGW2024positive}*{Hypothesis~5.3}
with parameters $\bm{x}_0 = t^{-\tilde{A}}\bm{y}$ and $L = 3\tilde A-\alpha$.
\end{lemma}

\begin{proof}
This is clear by the definition \eqref{eq:RA} of $R_{A,\alpha}$.
\end{proof}

For the next two lemmas, let $\Gamma \in \ZZ$ and $\bm{w}\in K_\infty^6$.

\begin{lemma}
\label{LEM:old-general-int-estimate}
We have $J^\Gamma_{F,w}(\bm{w})\ll (1+|\bm{w}|)^{-2}$.
Moreover, if $\norm{\bm{w}}$ is sufficiently large, then $J_{F,w}(\gamma, \bm{w})=0$ unless $|\gamma|\asymp |\bm{w}|$.
\end{lemma}

\begin{proof}
Immediate from \eqref{box-integrals},
Lemma~\ref{LEM:w_A-is-nice},
and \cite{BGW2024positive}*{Lemma~5.5}.
\end{proof}

We note that \cite{BGW2024positive}*{Lemma~5.5} was proven for arbitrary $L\ge 0$, whereas \cite{BGW2024positive}*{Lemma~5.8} was only proven for $L=0$.
Therefore, the following lemma requires a bit more work.

\begin{lemma}
\label{LEM:new-vanishing-for-small-dual-form}
We have $J^\Gamma_{F,w}(\bm{w})=0$ unless $\abs{F^\ast(\bm{w})}\ll 1+\norm{\bm{w}}^{\deg{F^\ast}-1}$. 
\end{lemma}

\begin{proof}
We roughly follow \cite{BGW2024positive}*{proof of Lemma~5.8}.
As there we may assume $\norm{\bm{w}} \gg 1$, with an implied constant as large as we wish.
Then by Lemma~\ref{LEM:old-general-int-estimate}, we have $J^\Gamma_{F,w}(\bm{w})=0$
unless $1 \ll \norm{\bm{w}} \ll \hat\Gamma$, in which case
\begin{equation*}
J^\Gamma_{F,w}(\bm{w})
= \int_{\abs{\gamma}\asymp \norm{\bm{w}}} J_{F,w}(\gamma,\bm{w}) \,\dd\gamma.
\end{equation*}

As noted in the proof of Lemma~\ref{LEM:integral-scale-invariance}, we have $w(\bm{x}) = w(\lambda\bm{x})$ for all $\lambda\in 1+\TT$.
The change of variables $\bm{x}\mapsto \lambda^{-1/2}\bm{x}$ preserves the Haar measure $\dd\bm{x}$ on $K_\infty^6$, whence
\begin{equation*}
J_{F,w}(\gamma,\bm{w})
= J_{F,w}(\lambda^{-3/2}\gamma,\lambda^{-1/2}\bm{w})
= J_{F,w}(\lambda^{-1/2}\gamma/\lambda,\lambda^{-1/2}\bm{w}).
\end{equation*}
Letting $S\belongs K_\infty^\times$ be a complete set of representatives for the quotient group $K_\infty^\times/(1+\TT)$, and writing $\gamma=\gamma_0\lambda$ with $\gamma_0\in S$ and $\lambda\in 1+\TT$, we deduce that
\begin{equation}
\label{emph-1}
J^\Gamma_{F,w}(\bm{w})
= \sum_{\substack{\gamma_0\in S\\  \abs{\gamma_0}\asymp \norm{\bm{w}}}} \abs{\gamma_0}
\int_\TT J_{F,w}(\lambda^{-1/2}\gamma_0,\lambda^{-1/2}\bm{w}) \,\dd\mu,
\end{equation}
where $\lambda\defeq 1+\mu$.

By stationary phase, in the form  \cite{BGW2024positive}*{Lemma~5.2},
applied to each weight $\omega_{\bm{y}}$ in \eqref{box-integrals} after
a linear change of variables
with bounded coefficients,
we have
\begin{equation}
\begin{split}
\label{emph-2}
J_{F,w}(\lambda^{-1/2}\gamma_0,\lambda^{-1/2}\bm{w})
&= q^{n\tilde{A}} \sum_{\bm{y}\in R_{A,\alpha}^2}
J_{F,\omega_{\bm{y}}}(t^{3\tilde{A}}\lambda^{-1/2}\gamma_0, t^{\tilde{A}}\lambda^{-1/2}\bm{w}) \\
&= \int_\Omega w(\bm{x}) \psi(\lambda^{-1/2}\Phi(\bm{x})) \,\dd\bm{x},
\end{split}
\end{equation}
where $\Phi(\bm{x}) \defeq \gamma_0 F(\bm{x})+\bm{w}\cdot \bm{x}$,
and where $\Omega\subseteq K_\infty^n$ is a region such that
\begin{equation*}
\Omega \subseteq \set{\bm{x}\in \supp(w):
\abs{\nabla{\Phi}(\bm{x})} \ll \max(1,\abs{\gamma_0})^{1/2}}.
\end{equation*}
Since $\abs{\gamma_0}\asymp \norm{\bm{w}} \gg 1$, this means
\begin{equation}
\label{final-Omega}
\Omega \subseteq \set{\bm{x}\in \supp(w):
\abs{\nabla{\Phi}(\bm{x})} \ll \abs{\gamma_0}^{1/2}}.
\end{equation}

The function $\Phi$ depends on $\gamma_0$, but not on $\lambda$.
Moreover, $\Phi(\bm{x}) \ll_w \norm{\bm{w}}$ for $\bm{x}\in \supp(w)$, so $\psi(\lambda^{-1/2}\Phi(\bm{x}))$ depends only on the first $(\log_q\norm{\bm{w}}) + O_w(1)$ terms of the power series expansion $\lambda^{-1/2} = \sum_{k\ge 0} \binom{-1/2}{k} \mu^k$.
The coefficients of this power series in $\mu$ all have absolute value $\le 1$, so by \cite{BGW2024positive}*{Lemma~5.2} for a suitable polynomial $G$, we get
\begin{equation}
\label{emph-3}
\int_\TT \psi(\lambda^{-1/2}\Phi(\bm{x})) \,\dd\mu = 0,
\end{equation}
unless $\abs{\Phi(\bm{x})} \le 1$.
Here $\norm{\bm{w}}$ is allowed to be unbounded.
Indeed, the statement of \cite{BGW2024positive}*{Lemma~5.2} depends on the height $H_G$, but not on the degree $\deg{G}$.

Combining \eqref{emph-1}, \eqref{emph-2}, and \eqref{emph-3}, we conclude that $J^\Gamma_{F,w}(\bm{w}) = 0$,
unless there exist an element $\gamma_0\in S$, and a point $\bm{x}\in \Omega$,
such that $\abs{\gamma_0}\asymp \norm{\bm{w}}$ and $\abs{\Phi(\bm{x})} \le 1$.
Arguing via \eqref{final-Omega} and \cite{BGW2024positive}*{Lemma~5.7} as in the final two paragraphs of \cite{BGW2024positive}*{proof of Lemma~5.8}, it follows that $\abs{F^\ast(\bm{w})}\ll 1+\norm{\bm{w}}^{\deg{F^\ast}-1}$, as desired.
\end{proof}

\section{Ratios analysis}
\label{SEC:new-ratios-stuff}
In this section we initiate the proof of Theorem 
\ref{Thm: Manin} and 
  collect together some of the estimates coming from the Ratios Conjecture that will be useful in this endeavour. 
Recall the definition \eqref{eq:fels} of the counting function $N_w(P)$, where $w$ is the weight function defined in Definition \ref{DEFN:A-skew-weights}.
  Applying 
\cite{BGW2024positive}*{Eq.~(2.9)}, the circle method leads to the expression
$$
    N_w(P)=|P|^n\sum_{\substack{r\in\Omon\\ |r|\leq \widehat{Q}}}|r|^{-n}\sum_{\bm{c}\in \OK^n}S_r(\bm{c})I_r(\bm{c}),
$$
for suitable exponential sums $S_r(\bm{c})$ and oscillatory integrals 
$I_r(\bm{c})$. (In fact, we have 
$I_r(\bm{c})=J_{F,w}^\Gamma(P\bm{c}/r)$, with $\Gamma=-\deg(r)-Q$.)
We make the decomposition
    \begin{equation}\label{eq:stein}
    N_w(P)=M(P)+E_1(P)+E_2(P),       
    \end{equation}
where $M(P)$ is the contribution from $\bm{c}=\bm{0}$, 
$E_1(P)$ is the contribution from $\bm{c}$ for which 
$F^*(\bm{c})\neq 0$, and finally, 
$E_2(P)$ is the contribution from $\bm{c}$ for which 
$F^*(\bm{c})= 0$.

For the remainder of this section our goal will be to  harness the Ratios Conjecture to ensure that $E_1(P)$ makes a negligible contribution to $N_w(P)$.

Let $V$ and $V_{\bm{c}}$ be the $K$-varieties in $\PP^{5}_K$ defined by $F(\bm{x})=0$ and $F(\bm{x})=\bm{c}\cdot\bm{x}=0$, respectively. 
Let $L(s,V) = L(s,H^4_\ell(V)/H^4_\ell(\PP^5))$
and let
$
L(s,\bm{c}) = L(s,H^3_\ell(V_{\bm{c}})),
$
for $\bm{c}\in \mathcal{S}_1 = \{\bm{c}\in \OK^n: F^\ast(\bm{c})\neq 0\}$.
As in our previous work, we write  $\mu_{\bm{c}}(r)$ for  the $r$th coefficient of the Euler product $L(s,\bm{c})^{-1} = \prod_\varpi L_\varpi(s,\bm{c})^{-1}$. 
The following is a minor variant of \cite{BGW2024positive}*{Proposition~3.5}.

\begin{proposition}
\label{PROP:(LocAvSp)}
Let $\bm{a}\in \OK^6$ and $d,r\in \Omon$.
Let $\EE^{\bm{a},d}_{\bm{c}\in S}[f]$ be the average of $f$ over $\set{\bm{c}\in S: \bm{c}\equiv\bm{a}\bmod{d}}$ (assuming this set is nonempty).
The limit
\begin{equation*}
\bar{\mu}_{F,1}^{\bm{a},d}(r)
\defeq \lim_{Z\to \infty}
\EE^{\bm{a},d}_{\bm{c}\in \mathcal{S}_1:\, \norm{\bm{c}}\le \hat Z}
[\mu_{\bm{c}}(r)]\end{equation*}
exists.
Moreover, $\bar{\mu}_{F,1}^{\bm{a},d}(r)\bar{\mu}_{F,1}^{\bm{a},d}(r')
= \bar{\mu}_{F,1}^{\bm{a},d}(rr')$ if $\gcd(r,r')=1$.

Now let $\varpi\in \Omon$ be a prime, and let $l\ge 0$ be an integer.
Then
\begin{equation}
\label{INEQ:GRC-bound-for-averages-mu-bar}
\bar{\mu}_{F,1}^{\bm{a},d}(\varpi^l) \ll_\eps \abs{\varpi}^{l\eps}.
\end{equation}
Furthermore, if $\varpi\nmid d$, then
$\bar{\mu}_{F,1}^{\bm{a},d}(\varpi^l)
= \bar{\mu}_{F,2}(\varpi^l,1)$.
\end{proposition}

\begin{proof}
The final sentence is obvious.
Everything else, up to and including \eqref{INEQ:GRC-bound-for-averages-mu-bar},
is proven just as in the proof of \cite{BGW2024positive}*{Proposition~3.5}, 
using \cite{BGW2024positive}*{Lemma~3.4}, the bound $\mu_{\bm{c}}(r) \ll_\eps \abs{r}^\eps$, and the Chinese remainder theorem.
\end{proof}

Informally, Proposition~\ref{PROP:(LocAvSp)},
combined with \cite{BGW2024positive}*{Eq.~(3.5)},
tells us
\begin{equation*}
\sum_{r\in \Omon} \bar{\mu}_{F,1}^{\bm{a},d}(r) \abs{r}^{-s}
\approx \prod_\varpi (1 + \lambda_V(\varpi)\abs{\varpi}^{-s-1/2} + \abs{\varpi}^{-2s})
\approx L(s+\tfrac12,V) \zeta_K(2s).
\end{equation*}
In fact, the Euler product 
\begin{equation}\label{eq:def-A1}
A_{F,1}^{\bm{a},d}(s) \defeq \zeta_K(2s)^{-1} L(s+\tfrac12,V)^{-1} \sum_{r\in \Omon} \bar{\mu}_{F,1}^{\bm{a},d}(r) \abs{r}^{-s}
\end{equation}
converges absolutely for $\Re(s) \ge \frac13 + \eps$, for any $\eps>0$, and satisfies
\begin{equation}
\label{INEQ:A_F,1-bound}
A_{F,1}^{\bm{a},d}(s) \ll_\eps \abs{d}^\eps.
\end{equation}
The expression $A_{F,1}^{\bm{a},d}$ appears as the ``leading constant'' in the Ratios Conjecture~\ref{CNJ:RA1}~(RA1).
The Ratios Recipe \cite{conrey2008autocorrelation}*{\S~5.1}, directly adapted to function fields as in \cite{andrade2014conjectures}, produces (RA1), even with a power-saving error term $O(\hat Z^{-\delta})$ independent of $\beta$ for $\beta\le \delta$, say.

Let $A_{F,1}^{\bm{a},d}(s)$ be defined as in 
\eqref{eq:def-A1}, in terms of certain local averages $\bar{\mu}_{F,1}^{\bm{a},d}(r)$.
\begin{conjecture}[RA1]
\label{CNJ:RA1}
Fix a real $M\ge 0$.
Let $\bm{a}\in \OK^6$ and $d\in \Omon$ with $\norm{\bm{a}}, \abs{d} \le \hat M$.
Let $\bm{b}\in K_\infty^6$ with $\norm{\bm{b}}\le 1$.
There exists $\beta=\beta_M(Z)\in [0,1]$ such that
if $$s=\beta+\sigma(Z)+i\ts,$$
then uniformly over $Z\in \NN$ and $\ts\in \RR$, we have
$$
\sum_{\substack{\bm{c}\in \mathcal{S}_1 \\
\bm{c}\equiv \bm{a}\bmod{d} \\
\norm{\bm{c} - t^Z\bm{b}}\le \hat Z/\hat M}}
\Phi^{\bm{c},1}(s)
= \sum_{\substack{\bm{c}\in \mathcal{S}_1 \\
\bm{c}\equiv \bm{a}\bmod{d} \\
\norm{\bm{c} - t^Z\bm{b}}\le \hat Z/\hat M}}
(1 + O(g \hat Z^{-3\beta})) \, A_{F,1}^{\bm{a},d}(s),
$$
for some function $g=g_M(Z) \to 0$ as $Z\to \infty$.
\end{conjecture}

As in the case of (R2), it is expected that $\beta$ in (RA1) can be taken to be any small constant, including $0$.
When $M=0$, (RA1) follows easily from (R2), with $g = \hat Z^{-3\beta}$, if $\beta>0$.
It would be interesting to determine whether there is a similar implication for general $M\ge 0$.

We now build on (RA1).
Let $\bar{a}_{F,1}^{\bm{a},d}(r)$ be the $r$th coefficient of $A_{F,1}^{\bm{a},d}(s)$.

\begin{proposition}
\label{PROP:RA1'E}
Assume (RA1).
Let $M\ge 0$, $\norm{\bm{a}}, \abs{d} \le \hat M$, and $\norm{\bm{b}}\le 1$.
Let $Z,R\in \ZZ$ with $R\le 3Z$.
Then
$$
\sum_{\substack{\bm{c}\in \mathcal{S}_1 \\
\bm{c}\equiv \bm{a}\bmod{d} \\
\norm{\bm{c} - t^Z\bm{b}}\le \hat Z/\hat M}}
\sum_{r\in \Omon}
\bm{1}_{\abs{r} = \hat R} \cdot (a_{\bm{c},1}(r) - \bar{a}_{F,1}^{\bm{a},d}(r))
\ll \hat M^\eps g_M(Z) \sum_{\substack{\bm{c}\in \mathcal{S}_1 \\
\bm{c}\equiv \bm{a}\bmod{d} \\
\norm{\bm{c} - t^Z\bm{b}}\le \hat Z/\hat M}} \hat R^{1/2}.
$$
\end{proposition}

\begin{proof}
Simply plug in the identity $\bm{1}_{\abs{r} = \hat R} = \EE_{\Re(s)=\beta+\sigma(Z)}[(\hat R/\abs{r})^s]$,
the bound \eqref{INEQ:A_F,1-bound},
and the inequality $\hat R^{\beta+\sigma(Z)} \hat Z^{-3\beta} \le \hat R^{1/Z} \hat R^{1/2} \le q^3 \hat R^{1/2}$.
\end{proof}

In fact, the contribution from $\bar{a}_{F,1}^{\bm{a},d}(r)$ is small, in terms of the square root $\hat R^{1/2}$.

\begin{proposition}
\label{PROP:mollified-barF,1-decay}
Let  $d\in \Omon$ and $R\in \ZZ$. Then
$$\sum_{r\in \Omon}
\bm{1}_{\abs{r} = \hat R} \cdot \bar{a}_{F,1}^{\bm{a},d}(r)
\ll_\eps \abs{d}^\eps \hat R^{1/3+\eps}.
$$
\end{proposition}

\begin{proof}
On writing $\bm{1}_{\abs{r} = \hat R} = \EE_{\Re(s)=\frac13+\eps}[(\hat R/\abs{r})^s]$, we get
\begin{equation*}
\sum_{r\in \Omon}
\bm{1}_{\abs{r} = \hat R} \cdot \bar{a}_{F,1}^{\bm{a},d}(r)
= \EE_{\Re(s)=\frac13+\eps}[\hat R^s A_{F,1}^{\bm{a},d}(s)]
\ll_\eps \abs{d}^\eps \hat R^{1/3+\eps},
\end{equation*}
where the first step is justified by the absolute convergence of $A_{F,1}^{\bm{a},d}(s)$
and the second step is justified by \eqref{INEQ:A_F,1-bound}.
\end{proof}

The next step is an analogue of \cite{wang2023ratios}*{Conjecture~7.14}.

\begin{proposition}
\label{PROP:RA1'E'}
Assume (R2) and (RA1).
Let $Z,R\in \ZZ$ with $R\le 3Z$.
Let $M\in [0,R]$ and $\abs{d} \le \hat M$.
Partition the box $\{\cc\in \OK^n: \norm{\cc} \le \hat Z\}$ into sets
$$\{\cc\in \OK^n: \bm{c}\equiv \bm{a}\bmod{d},
\; \norm{\bm{c} - t^Z\bm{b}}\le \hat Z/\hat M\},$$
indexed by some set $\mscr{P}=\mscr{P}(Z,d,M)$ of pairs $(\bm{a},\bm{b})$
with $\norm{\bm{a}} < \abs{d}$ and $\norm{\bm{b}} \le 1$.
Then
$$
\sum_{(\bm{a},\bm{b})\in \mscr{P}}\,
\Biggl\lvert
\sum_{\substack{\bm{c}\in \mathcal{S}_1 \\
\bm{c}\equiv \bm{a}\bmod{d} \\
\norm{\bm{c} - t^Z\bm{b}}\le \hat Z/\hat M}}
\sum_{r\in \RcG} \bm{1}_{\abs{r} = \hat R} \cdot S^\natural_r(\cc)
\Biggr\rvert
\ll \left(f_0(M) + f_{1,M}(Z)\right)
\hat Z^6 \hat R^{1/2},
$$
for some functions $f_0(A)$ and $f_{1,M}(A)$ tending to $0$ as $A\to \infty$.
\end{proposition}

\begin{proof}
By the $\eps=1$ case of \cite{BGW2024positive}*{Conjecture~4.5},
the quantity to be bounded is certainly $O(\hat Z^6 \hat R^{1/2})$.
We go further, obtaining cancellation over $\bm{c}$, 
by introducing (RA1) and \cite{BGW2024positive}*{Lemma~3.4} as additional inputs.
Let $M_0\ge 0$ be an auxiliary parameter.
In the notation of the proof of \cite{BGW2024positive}*{Conjecture~4.5}, with $\beta=1$, we have
\begin{equation*}
\sum_{r\in \RcG} \bm{1}_{\abs{r} = \hat R} \cdot S^\natural_r(\cc)
= \sum_{R_1+R_2+R_3=R} \prod_{1\le j\le 3} \Sigma^{\bm{c},j}(R_j)
\end{equation*}
and $\sum_{\bm{c}\in \mcal{S}_1:\, \norm{\bm{c}}\le \hat Z} \prod_{1\le j\le 3} \abs{\Sigma^{\bm{c},j}(R_j)} \ll \hat Z^6 \hat R^{1/2} \hat R_2^{-9/40} \hat R_3^{-2/15}$, thanks to (R2).
The total contribution from $R_2+R_3 \ge M_0$ is therefore $\ll \hat M_0^{-2/15} \hat Z^6 \hat R^{1/2}$.

On the other hand, if $R-R_1 < M_0$, then by \cite{BGW2024positive}*{Lemma~3.4},
there exist a modulus $f(M_0)\in \Omon$, and an exceptional set $\mscr{E}(M_0)\belongs \OK^6$ of density $\mathfrak{d}(M_0) = o_{M_0 \to \infty}(1)$ defined by congruence conditions modulo $f(M_0)$,
such that for all $\bm{c}\notin \mscr{E}(M_0)$, the quantity $\sum_{R_2+R_3 = R-R_1} \prod_{2\le j\le 3} \Sigma^{\bm{c},j}(R_j)$ depends at most on the residue class $$\bm{c} \bmod{f(M_0)}.$$
Since $\Sigma^{\bm{c},j}(R_j) \ll \hat R_j^{1+\eps}$ trivially,
we conclude by Propositions~\ref{PROP:RA1'E} and~\ref{PROP:mollified-barF,1-decay},
applied with parameters ``$(M,d) \defeq (M + \deg{f(M_0)}, f(M_0)d)$'',
that
\begin{equation*}
\sum_{\substack{\bm{c}\in \mathcal{S}_1 \setminus \mscr{E}(M_0) \\
\bm{c}\equiv \bm{a}\bmod{d} \\
\norm{\bm{c} - t^Z\bm{b}}\le \hat Z/\hat M}}
\sum_{\substack{R_1+R_2+R_3=R \\ R_2+R_3<M_0}} \prod_{1\le j\le 3} \Sigma^{\bm{c},j}(R_j)
\ll \hat M_0^{1+\eps} T \sum_{\substack{\bm{c}\in \mathcal{S}_1 \setminus \mscr{E}(M_0) \\
\bm{c}\equiv \bm{a}\bmod{d} \\
\norm{\bm{c} - t^Z\bm{b}}\le \hat Z/\hat M}} \hat R^{1/2},
\end{equation*}
where $$T \defeq \hat M^\eps \abs{f(M_0)}^\eps g_{M + \deg{f(M_0)}}(Z)
+ \abs{f(M_0)d}^\eps \hat R^{\eps - 1/6}.$$
Here, the first term comes from the right-hand side of Proposition~\ref{PROP:RA1'E}, whereas 
the second term of $T$ comes from the right-hand side of Proposition~\ref{PROP:mollified-barF,1-decay}.

Summing over $(\bm{a},\bm{b})\in \mscr{P}$, then adding in the $R_2+R_3 \ge M_0$ contribution, we get
\begin{equation*}
\sum_{(\bm{a},\bm{b})\in \mscr{P}}\,
\Biggl\lvert
\sum_{\substack{\bm{c}\in \mathcal{S}_1\setminus \mscr{E}(M_0) \\
\bm{c}\equiv \bm{a}\bmod{d} \\
\norm{\bm{c} - t^Z\bm{b}}\le \hat Z/\hat M}}
\sum_{R_1+R_2+R_3=R} \prod_{1\le j\le 3} \Sigma^{\bm{c},j}(R_j)
\Biggr\rvert
\ll T' \hat Z^6 \hat R^{1/2},
\end{equation*}
where $T' \defeq \hat M_0^{-2/15} + \hat M_0^{1+\eps} T$.
Yet by \cite{BGW2024positive}*{Conjecture~4.5, proven under (R2)},
and H\"{o}lder's inequality in the form $1 = \frac{1-\eps}{2-\eps} + \frac{1}{2-\eps}$, we have
\begin{equation*}
\sum_{\substack{\bm{c}\in \mathcal{S}_1 \cap \mscr{E}(M_0) \\ \norm{\bm{c}}\le \hat Z}}\,
\biggl\lvert
\sum_{r\in \RcG} \bm{1}_{\abs{r} = \hat R} \cdot S^\natural_r(\cc)
\biggr\rvert
\ll \mathfrak{d}(M_0)^{(1-\eps)/(2-\eps)} (\hat Z + \abs{f(M_0)})^6 \hat R^{1/2}.
\end{equation*}

Therefore, if $T'' \defeq T' + \mathfrak{d}(M_0)^{(1-\eps)/(2-\eps)} (1 + \abs{f(M_0)}/\hat Z)^6$, then
\begin{equation*}
\sum_{(\bm{a},\bm{b})\in \mscr{P}}\,
\Biggl\lvert
\sum_{\substack{\bm{c}\in \mathcal{S}_1 \\
\bm{c}\equiv \bm{a}\bmod{d} \\
\norm{\bm{c} - t^Z\bm{b}}\le \hat Z/\hat M}}
\sum_{r\in \RcG} \bm{1}_{\abs{r} = \hat R} \cdot S^\natural_r(\cc)
\Biggr\rvert
\ll T'' \hat Z^6 \hat R^{1/2}.
\end{equation*}
Letting $M_0 \defeq \max\{A\in [0,M/12]: \abs{f(A)} \le \hat M\}$, and recalling that $M\le R$, we get
\begin{equation*}
\begin{split}
T'' &\ll \hat M_0^{-2/15}
+ \hat M^{3\eps+1/12} g_{2M}(Z)
+ \hat R^{4\eps-1/12}
+ \mathfrak{d}(M_0)^{(1-\eps)/(2-\eps)} (1 + \abs{f(M_0)}^6/\hat Z^6) \\
&= o_{M\to \infty}(1)
+ o_{M; Z\to \infty}(1)
+ o_{R\to \infty}(1)
+ o_{M\to \infty}(1)
+ o_{M; Z\to \infty}(1),
\end{split}
\end{equation*}
because $M_0 \to \infty$ whenever $M\to \infty$,
and $\mathfrak{d}(M_0) \to 0$ whenever $M_0 \to \infty$.
This suffices, since $o_{R\to \infty}(1) \le o_{M\to \infty}(1)$.
\end{proof}

By the local constancy result \eqref{EQN:local-constancy},
and our dyadic bounds in \cite{BGW2024positive}*{\S~8},
we can now prove an analogue of \cite{wang2023ratios}*{Theorem~10.7},
going beyond \cite{BGW2024positive}*{Proposition~8.1}.

\begin{theorem}\label{Thm: E1Contri}
Assume (R2) and (RA1).
Then $E_1(P) = o_{w }(\abs{P}^3)$, as $\abs{P}\to \infty$.
\end{theorem}

\begin{proof}
The idea is to recycle our work from \cite{BGW2024positive}*{\S~8} as much as possible,
and only then to use \eqref{EQN:local-constancy} and Proposition~\ref{PROP:RA1'E'} in the remaining ranges.

First, we claim that for $0\le Y_1+Y_2=Y\le Q$, we have
\begin{equation}
\label{old-E1-bound}
\Sigma(Y_1,Y_2) \defeq
\abs{P}^n \sum_{\norm{\bm{c}}\ll \abs{P}^{1/2}} I_{t^Y}(\bm{c}) \hat Y^{(1-n)/2}
\sum_{\substack{r_1\in \RcG\\ \abs{r_1}=\hat Y_1}} S^\natural_{r_1}(\bm{c})
\sum_{\substack{r_2\in \RcB\\ \abs{r_2}=\hat Y_2}} S^\natural_{r_2}(\bm{c})
\ll \frac{\abs{P}^3}{\hat D^\omega}
\end{equation}
for some small constant $\omega>0$, where
\begin{equation}
\hat D \defeq \max\left\{\frac{\abs{P}^{3/2}}{\hat Y},
\hat Y_2\right\}
\ll \abs{P}^{3/2}.
\end{equation}
Indeed,
given the integral estimates Lemmas~\ref{LEM:old-general-int-estimate} and~\ref{LEM:new-vanishing-for-small-dual-form} (replacing \cite{BGW2024positive}*{Lemmas~5.5 and~5.8}),
our work in \cite{BGW2024positive}*{\S~8} leads to the following bounds.
\begin{enumerate}
\item By \cite{BGW2024positive}*{\S~8.1},
the contribution to $\Sigma(Y_1,Y_2)$ from $\norm{\bm{c}}\ll \abs{P}^{1/2-\delta}$ is
$$\ll \abs{P}^{3n/4-3/2+\eps-\delta(1+n/2+\eps)/2}
= \abs{P}^{3+\eps-\delta(4+\eps)/2}
\ll \abs{P}^3/\hat D^\omega,$$
unconditionally, provided $\eps,\omega \ll \delta$.

\item In the notation of \cite{BGW2024positive}*{\S~8.2},
if $\hat Y_2^{n+\eps} \le \hat W^{\alpha/2}$, then
the contribution to $\Sigma(Y_1,Y_2)$ from $\norm{\bm{c}}\gg \abs{P}^{1/2-\delta}$ is,
by \cite{BGW2024positive}*{final display before \emph{The case $\hat Y_2^{n+\eps} > \hat W^{\alpha/2}$}},
$$\ll \abs{P}^{3n/4-3/2}
(\hat Y_2^{-n/2} (\hat Y/\abs{P}^{3/2})^{\alpha/2}
+ \hat Y_2^{-n/2} \abs{P}^{-\alpha(1+\eps)/4})
\ll \abs{P}^3/\hat D^\omega,$$
under (R2), provided $\omega\ll \alpha$.

\item In the notation of \cite{BGW2024positive}*{\S~8.2},
if $\hat Y_2^{n+\eps} > \hat W^{\alpha/2}$, then
the contribution to $\Sigma(Y_1,Y_2)$ from $\norm{\bm{c}}\gg \abs{P}^{1/2-\delta}$ is,
by \cite{BGW2024positive}*{antepenultimate and penultimate displays of \S~8},
$$\ll \abs{P}^{3n/4-3/2} \hat Y_2^{-\eta'/2}
((\hat Y/\abs{P}^{3/2})^\beta
+ \abs{P}^{\beta(\eps-1)/2})
\ll \abs{P}^3/\hat D^\omega,$$
under (R2), provided $\omega\ll \min\{\eta',\beta\}$.
Here $\beta \defeq \eta'\alpha / (4(n+\eps))$.
\end{enumerate}
These three cases complete the proof of  \eqref{old-E1-bound}.

On the other hand, we may bound $\Sigma(Y_1,Y_2)$
by first fixing $r_2$, then fixing $I_{t^Y}(\bm{c})$ using \eqref{EQN:local-constancy},
and finally summing over $\bm{c}$ and $r_1$ using Proposition~\ref{PROP:RA1'E'}, with $\hat Z \asymp \abs{P}^{1/2}$.
This gives
\begin{equation*}
\Sigma(Y_1,Y_2)
\ll \abs{P}^{n-3} \left(1 + \frac{\abs{P}\hat C}{\hat Y}\right)^{1-n/2}
\hat Y^{(1-n)/2} \hat Y_2^{1 + (1+n)/2}
\left(f_0(M) + f_{1,M}(Z)\right)
\hat Z^6 \hat Y_1^{1/2}
\end{equation*}
provided $M\in [0,Y_1]$ and $\hat M \gg \hat D$ with a sufficiently large implied constant.
The conditions on $\hat M$ ensure, in particular, that $I_{t^Y}(\bm{c})$ is constant on the box $\norm{\bm{c} - t^Z\bm{b}}\le \hat Z/\hat M$,
for any given $\bm{b}$ with $\norm{\bm{b}}\le 1$.

Since $\hat Y_1\hat Y_2 = \hat Y \ge \abs{P}^{3/2}/\hat D$ and $\hat Y_2 \le \hat D$, the last display is
\begin{equation*}
\ll \abs{P}^{n-3} (\abs{P}^{3/2}/\hat D)^{1-n/2} \hat D^{(2+n)/2}
\left(f_0(M) + f_{1,M}(Z)\right)
\abs{P}^3,
\end{equation*}
which simplifies to $\abs{P}^3 \hat D^n \left(f_0(M) + f_{1,M}(Z)\right)$.
Let $D_0$ be a parameter with $\hat M\gg \hat D_0$.
Then
\begin{equation*}
\begin{split}
\sum_{Y_1,Y_2} \Sigma(Y_1,Y_2)
&\ll \sum_{\substack{Y_1,Y_2 \\ D\ge D_0}} \frac{\abs{P}^3}{\hat D^\omega}
+ \sum_{\substack{Y_1,Y_2 \\ D\le D_0}} \abs{P}^3 \hat D^n
\left(f_0(M) + f_{1,M}(Z)\right) \\
&\ll \frac{\abs{P}^3}{\hat D_0^{\omega/2}}
+ \hat D_0^\eps \abs{P}^3 \hat D_0^n
\left(f_0(M) + f_{1,M}(Z)\right).
\end{split}
\end{equation*}
We conclude that $\abs{E_1(P)} \le (2\hat D_0^{-\omega/2} + \hat D_0^{n+\eps} f_0(M)) \abs{P}^3$ for all $\abs{P}\gg_{D_0,M} 1$.
Taking $M\gg_{D_0} 1$, we then have $\abs{E_1(P)} \le 3\hat D_0^{-\omega/2} \abs{P}^3$ for all $\abs{P}\gg_{D_0} 1$.
Taking $D_0\to \infty$, we are finally done. 
\end{proof}

\section{Centre and dual variety}\label{Sec: CentreDual}

We refine \cite{BGW2024positive}*{Proposition~9.1} to an asymptotic for $M(P)$.
It will be convenient to make the explicit choice $Q=\lfloor -2\tilde{A} +3d/2\rfloor$. 
\begin{proposition}\label{Prop: Centercontri}
If $w=w_{A,\alpha}$,
then \[
M(P) = \sigma_\infty \Sing |P|^3 + O(\sigma_\infty |P|^3 \widehat{Q}^{-2/3+\varepsilon}),
\]
where $\Sing$ and $\sigma_\infty$ are the singular series and singular integral defined in \eqref{Eq: SingSer} and \eqref{Eq: singInt} respectively.
\end{proposition}

\begin{proof}
Recall that 
\[
M(P)=|P|^6\sum_{\substack{r\in \Omon \\ |r|\leq \hat{Q}}}|r|^{-6}I_r(\bm{0})S_r(\bm{0}).
\]
Inserting \eqref{Eq: DecompWeight} into the definition of the weight function $w$ and making the change of variables $\bm{x}=\bm{y}+t^{\alpha - 2\tilde{A}}\bm{z}$ with $\bm{y}\in R_{A,\alpha}$, it follows that 
\[
I_r(\bm{0})=\int_{|\theta|<|r|^{-1}\widehat{Q}^{-1}}\left(q^{3(\alpha -2\tilde{A})}\sum_{\bm{y}\in R_{A,\alpha}}\int_{\TT^3}\psi(\theta P^3 F_0(\bm{y}+t^{\alpha -2\tilde{A}}\bm{z}))\dd\bm{z}\right)^2\dd\theta. 
\]
By \cite[Lemma 5.2]{BGW2024positive} we may replace $\TT$ in the integral by 
\[
\Omega=\{\bm{z}\in\TT^3\colon |t^{\alpha -2\tilde{A}}\theta P^3\nabla F_0(\bm{y}+t^{\alpha -2 \tilde{A}}\bm{z})|\leq q^\alpha \max\{1, |\theta P^3|^{1/2}\}\},
\]
where we used the fact that $H_G\leq q^\alpha$ for $G(\bm{z})=F(\bm{y}+t^{\alpha -2\tilde{A}}\bm{z})$. As $\bm{y}\in R_{A,\alpha}$, in the notation of 
\eqref{eq:RA},
we have $1\leq |y_1|$, so that 
\[|\nabla F_0(\bm{y}+t^{\alpha-2\tilde A}\bm{z})|\geq |(y_1+t^{\alpha -2\tilde A}z_1)^2|\geq 1.
\]
In particular, if $\Omega \neq \emptyset$, then we must have 
\[
|\theta|\leq q^{2\tilde{A} }|P|^{-3}\max\{1,|\theta P^3|^{1/2}\}, 
\]
which is only possible if $|\theta|\leq q^{4\tilde{A}}|P|^{-3}$. As $|r|\leq \widehat{Q}$, we always have $|r|^{-1}\widehat{Q}^{-1}\geq \widehat{Q}^{-2}\geq q^{4\tilde{A}}|P|^{-3}$. Therefore, 
\begin{align*}
I_r(\bm{0})&=\int_{|\theta|\leq q^{4\tilde{A}}|P|^{-3}}\int_{K_\infty^6}w(\bm{x})\psi(\theta P^3 F(\bm{x}))\dd\bm{x}\dd\theta \\
& = |P|^{-3}\int_{|\theta|\leq q^{4\tilde{A}}}\int_{K_\infty^6}w(\bm{x})\psi(\theta F(\bm{x}))\dd\bm{x}\dd\theta\\
&=\sigma_\infty |P|^{-3},
\end{align*}
on recalling the definition~\eqref{Eq: singInt}.
Turning to the singular series defined in~\eqref{Eq: SingSer}, we have
\begin{align*}
\Biggl |\sum_{\substack{r\in \Omon \\ |r|\leq \widehat{Q}}}|r|^{-6}S_r(\bm{0}) - \Sing \Biggr | & \leq \sum_{\substack{r\in \Omon \\ |r|>\widehat{Q}}}|r|^{-6}|S_r(\bm{0})|\\
& \ll \widehat{Q}^{-2/3+\varepsilon},
\end{align*}
by  \cite{BGW2024positive}*{Lemma~9.2}.
\end{proof}

We proceed by proving an analogue of \cite{BGW2024positive}*{Proposition~10.1}, for our weight function. 

\begin{proposition}\label{Prop: LinSpacecontri}
If $w=w_{A,\alpha}$ and $F=x_1^3+\dots+x_6^3$, then
\begin{equation*}
E_2(P) = \sum_{L\in \Upsilon} \sum_{\bm{x}\in L\cap \OK^6} w(\bm{x}/P)
+ O_{A,\eps}(\abs{P}^{3-1/4+\eps}).
\end{equation*}
\end{proposition}

\begin{proof}
Every step of the proof of \cite{BGW2024positive}*{Proposition~10.1} directly generalises to the weight function $w=w_{A,\alpha}$, the details of which will not be repeated here.
(In fact the proof  works for any $w\in S(K_\infty^6)$ with $\bm{0}\notin \supp(w)$.
Indeed, the integral estimate \cite{BGW2024positive}*{Eq.~(10.13)} holds for all such $w$, by the arguments of \cite{glas2022question}*{\S~3},
whereas all other ingredients in \cite{BGW2024positive}*{\S~10} are valid for arbitrary $w\in S(K_\infty^6)$.)
\end{proof}
We now have all the ingredients at hand to complete our proof of Theorem~\ref{Thm: Manin}.
\begin{proof}[Proof of Theorem~\ref{Thm: Manin}]
On recalling the decomposition of $N_w(P)$ in \eqref{eq:stein}, the proof is an immediate consequence of Theorem~\ref{Thm: E1Contri} and Propositions~\ref{Prop: Centercontri} and \ref{Prop: LinSpacecontri}.
\end{proof}


\begin{thebibliography}{999} 

\bibitem{andrade2014conjectures}
J. C. Andrade and J. P. Keating, Conjectures for the integral moments and ratios of $L$-functions over function fields.  {\em J.\ Number Theory} {\bf 142} (2014), 102--148.

\bibitem{BGW2024positive}
T. Browning, J. Glas and V. Y. Wang,
Sums of three cubes over a function field.
{\em Preprint}, 2024. ({\tt arXiv:2402.07146}) 


\bibitem{conrey2008autocorrelation}
J. B. Conrey, D. W. Farmer and M. R. Zirnbauer, Autocorrelation of ratios of $L$-functions.
{\em Commun.\ Number Theory Phys.} {\bf 2} (2008), 593--636. 



\bibitem{diaconu2019admissible}
S. Diaconu,
{\em On admissible integers of cubic forms}. Senior thesis, Princeton University, 2019.

\bibitem{glas2022question}
J. Glas and L. Hochfilzer, On a question of Davenport and diagonal cubic forms over $\FF_q(t)$.
{\em Preprint}, 2022. ({\tt arXiv:2208.05422}) 

\bibitem{HBP}
D. R. Heath-Brown and S. J. Patterson,
The distribution of Kummer sums at prime arguments.
{\em J.\ reine  angew.\ Math.} {\bf 310} (1979),
111--130.



\bibitem{IR} K. Ireland and M. Rosen, {\em A classical introduction to
  modern number theory}. 2nd ed., Springer-Verlag, 1990.




\bibitem{vaserstein1991sums}
L. N. Vaserstein, Sums of cubes in polynomial rings.
{\em Math.\ Comp.} {\bf 56} (1991), 349--357.

\bibitem{wang2022thesis}
V. Y. Wang,
{\em Families and dichotomies in the circle method}. Ph.D. thesis, Princeton University, 2022.

\bibitem{wang2023ratios}
\bysame,
Sums of cubes and the {Ratios} {Conjectures}.
{\em Preprint}, 2023. 
({\tt arXiv:2108.03398})

\end{thebibliography}
\end{document}